\documentclass[10.5pt]{amsart}
\usepackage[utf8]{inputenc}
\usepackage[english]{babel}
\usepackage{amsmath,amsfonts,amssymb,relsize,amsthm}
\usepackage{dsfont}
\usepackage{esint}
\usepackage{float}
\usepackage{algorithmic}
\usepackage[ruled,longend]{algorithm2e}
\usepackage[font=small,skip=0pt]{caption}
\usepackage{amsaddr}

\usepackage{graphicx}
\usepackage{tikz}
\usetikzlibrary{matrix}
\usetikzlibrary{arrows,positioning}
\usepackage{color}
\usepackage{caption}
\usepackage{subcaption}
\usetikzlibrary{3d}
\usepackage{xifthen}
\usetikzlibrary{decorations.pathmorphing}

\usepackage{hyperref} %%Warning: when you first run your tex
\usepackage[active]{srcltx} %%allows you to jump from your editor
\usepackage{color,soul} %%allows you to use the \st{...} command, which
\usepackage[a4paper,hmargin=0.75in,vmargin=1.1in]{geometry}

\newtheorem{theorem}{Theorem}[section]

\newtheorem{proposition}[theorem]{Proposition}

\theoremstyle{definition}
\newtheorem{definition}[theorem]{Definition}
\newtheorem{remark}{Remark}
\newtheorem{assumption}{Assumption}

\numberwithin{equation}{section}
\newtheorem{claim}[theorem]{Claim}

\def\f{{\mathcal{F}}}

\def\L{{\mathcal{L}}}
\def\M{{\mathcal{M}}}

\def\D{{\mathcal{D}}}

\def\X{{\mathcal{X}}}

\def\R{{\mathbb{R}}}
\def\Rd{{\mathbb{R}^d}}
\def\C{{\mathcal{C}}}

\def\N{{\mathbb{N}}}
\def\P{{\mathbb{P}}}

\def\Mp{{\mathcal{M}_{\text{pt}}}}

\def\o{{\omega}}
\def\eps{{\varepsilon}}
\def\nuvec{{\boldsymbol{\nu}}}
\def\phivec{{\boldsymbol{\phi}}}
\def\onevec{{\boldsymbol{1}}}

\def\E{{\mathbb{E}}}

\newcommand*{\myprime}{^{\prime}\mkern-1.2mu}

\DeclareMathOperator{\Exp}{Exp}
\DeclareMathOperator{\Unif}{Unif}
\DeclareMathOperator{\Randint}{RandomChoice}
\DeclareMathOperator{\Progeny}{Progeny}
\DeclareMathOperator{\Simulate}{Simulate}

\newcommand{\indic}[1]{\mathds{1}_{\left\lbrace #1 \right\rbrace}}

\newcommand{\nn}{\nonumber}

\newcommand\norm[1]{\left\lVert#1\right\rVert}

\usepackage{tikz}
\usetikzlibrary{automata}
\usetikzlibrary{positioning}  %                 ...positioning nodes
\usetikzlibrary{arrows}       %                 ...customizing arrows
\tikzset{node distance=4.5cm, % Minimum distance between two nodes. Change if necessary.
         every state/.style={ % Sets the properties for each state
           semithick,
           fill=gray!10},
         initial text={},     % No label on start arrow
         double distance=4pt, % Adjust appearance of accept states
         every edge/.style={  % Sets the properties for each transition
         draw,
           ->,>=stealth',     % Makes edges directed with bold arrowheads
           auto,
           semithick}}

\newcommand{\ds}{\displaystyle}
\usepackage{cleveref}
\usepackage{todonotes}
\usepackage{lineno}
\usepackage[
    type={CC},
    modifier={by-nc-sa},
    version={4.0},
]{doclicense}
\usepackage{csquotes}
\crefname{hyp}{hypothesis}{hypotheses}
\crefname{thm}{theorem}{theorems}
\crefname{lem}{lemma}{lemmas}
\crefname{cor}{corollary}{corollaries}
\crefname{prop}{proposition}{propositions}
\Crefname{theorem}{Theorem}{Theorems}

\def\eps{{\varepsilon}}
\def\Lep{\mathcal{L}_\eps}
\newcommand{\opLep}[1]{\Lep[\,{#1}\,]}

\newcommand{\JC}[1]{\textcolor{black}{#1}}

\title{{\bf Group Dispersal Modelling revisited}}

\author[Mario Ayala, Jerome Coville and Samuel Soubeyrand]{Mario Ayala $^{1,2}$, Jerome Coville $^{2,3}$, Samuel Soubeyrand $^{2}$ }
\address{${~}^1$\, {\bf Technische Universität München,} \\ {\bf School of Computation Information and Technology,} \\ Boltzmannstr. 3, 85748 Garching, Germany \\ ~\\ ${~}^2$\, {\bf INRAE, BioSP,} \\ 228 route de l'a\'erodrome, 84914 Avignon, France.\\ ~\\
${~}^{3}$\, { \bf Institut Camille Jordan, UMR 5208 - Universit\'e of Lyon 1}
}

\email{mario.ayala@tum.de}
\thanks{Part of this work was done while M. Ayala was at INRAE supported by the project ARCHIV of the Agence Nationale de la Recherche (ANR-18-CE32-0004)}
\email{jerome.coville@inrae.fr}
\email{samuel.soubeyrand@inrae.fr}

\begin{document}

\begin{abstract}
In this paper we revisit the notion of grouped dispersal that have been introduced by Soubeyrand and co-authors \cite{soubeyrand2011patchy} to model the simultaneous (and hence dependent) dispersal  of several propagules from a single source  in a homogeneous  environment. We built a time continuous  measure valued process that takes into account the main feature of a grouped dispersal and derive its infinitesimal generator. To cope with the mutligeneration aspect associated to the demography we introduce  two types of propagules in the description of  the population which is one of  the main innovations here.  We also provide a rigorous description of the process and its generator. We derive  as well, some large population asymptotics of the process unveilling the  degenerate ultra parabolic system of PDE satisfied by the density of population. Finally, we also show that such a PDE system has a non-trivial solution which is unique in a certain functional space.
\end{abstract}
\maketitle

\section{Introduction}
The notion of grouped dispersal has been introduced to analyse situations where several  entities move or are moved
together from one location to another one. This dispersal mechanism has been studied from an ecological perspective mainly for animals, e.g. juvenile tarantulas \cite{reichling2000} and female viscachas \cite{branch1993recruitment}, and for seeds eaten and transported by large mammals \cite{howe1989,takahashi2008} or birds \cite{pizo2001}. 
It also appears in epidemiological contexts to understand the creation of multiple secondary foci of infection of airborne pathogens \cite{soubeyrand2011patchy}. In particular, \cite{soubeyrand2011patchy} proposes a mechanical description of this process assuming a hierarchical structure of dependence on the ``propagules'' transport.

Namely, the Grouped Dispersal Model, GDM for short, introduced in \cite{soubeyrand2011patchy} relies on the following assumptions: 

\begin{enumerate}
    \item The numbers of propagules in each new group are independently generated from a counting distribution.
    \item The barycenters of groups are independently transported according to a dispersal kernel.
    \item Within each group, propagules follow independent Brownian motions, which are centered around the respective group barycenter.
    \item Each Brownian motion, representing the trajectory of a propagule, is stopped at a stopping time proportional to the distance between the source and the deposit location of the group's barycenter. 
\end{enumerate}

Starting with an initial source at $0$, the above assumptions lead  to consider a time discrete process where the numbers of groups and propagules by group are generated randomly with some prescribed law and that the position of the deposited propagule satisfies
$$
X_{jn} = X_j +B_{jn}(\nu \|X_j\|),
$$
where $X_j$ is the location of the barycenter of the group $j$, $X_{jn}$ is the location of the $n$-th propagule of the group $j$ and $B_{jn}$ is a Brownian motion followed by the propagule. The stopping time is proportional to the distance of the barycenter $\|X_j\|$ by a uniform constant factor $\nu$, such that the spatial extent of the final locations of the grouped propagules increases with the travelled distance.

This formulation allows  the formation of patchy patterns with the existence of secondary foci compatible with observed data. However, the resulting multigeneration process induces some additional constrains: in particular, all the propagules become sources at once, which prevents a natural identification of a continuous time version of the GDM process, and as well, makes difficult the  derivation of  large population asymptotics. 

In fact, due to these constrains, only preliminary modelling studies of grouped dispersal and its consequences  have been carried out in the literature. 
For instance, the conditions under which grouped dispersal might confer an evolutionary advantage and be selected are  investigated numerically in \cite{soubeyrand2015evolution} while the question of how efficient/evolutionary advantageous grouped dispersal is in the presence of an Allee effect  is discussed using an ad hoc meta-population model in  \cite{soubeyrand2017group}. 
Further questions regarding the specificity of the spatial or spatio-temporal population distributions/patterns generated by grouped dispersal with respect to those generated by independent dispersal,  are   addressed,  respectively, in \cite{soubeyrand2011patchy} for homogeneous environments  and in \cite{soubeyrand2014nonstationary} for heterogeneous environments.
These preliminary studies have raised  many important questions that, up to our knowledge, remain unanswered in the context of a large population regime, largely due to the lack of suitable analytical tools.

In this article, we revisit the notion of grouped dispersal, and  propose new modelling tools that in a sense generalise the approach and stochastic process defined in \cite{soubeyrand2011patchy}, and may allow us to investigate (more thoroughly and conveniently)  some of the questions listed above. In particular, our aim is to  set up a flexible approach describing continuous time group dispersal of particles for which large population asymptotics can be easily derived, and that can also be  extended to the description of the movement of a broader variety of geometric objects, e.g. disks, cylinders \cite{soubeyrand2014nonstationary}. 
To this end, we use the methodology introduced in \cite{fournier_microscopic_2004}, and in subsequent works (\cite{champagnat_invasion_2007, champagnat_individual_2008,meleard_stochastic_2015}), to build a measure-valued stochastic model representing the dynamics of the particles. This modelling framework offers significant advantages. The resulting measure-valued stochastic process is well-suited for computer simulations, enabling detailed and scalable analysis. Additionally, the framework opens the possibility of bringing mathematical tools from statistical mechanics, such as large deviation principles (LDP) or meta-stability techniques, facilitating rigorous analysis of the system's probabilistic structure. This dual capability of simulation and analytical tractability makes the framework a powerful tool for exploring a wide range of ecological and evolutionary dynamics, as the ones described above. We take advantage of these features and establish some large population asymptotics  of the constructed measure value GDM stochastic process and give a proper analysis of the ultra parabolic PDE system obtained.

\subsection*{Grouped dispersal Model}
Unlike the GDM  of \cite{soubeyrand2011patchy}, which assumes indistinguishable particles that are sources of other particles and also can travel, in order to make sense to a continuous-time measure-valued version of the GDM model we introduce two populations denoted by $\nu_s(t)$  and $\nu_p(t)$. These two populations represent respectively a population of seeds that travels through space, and a population of plants that become new sources of seeds. We assume  that these two populations live in the spatial domain $\bar{\X}$ which is the closure  of an open connected  subset $\X$ of $\R^d$ for some $d \geq 1$. 
Working as in  \cite{fournier_microscopic_2004,champagnat_invasion_2007, champagnat_individual_2008,meleard_stochastic_2015}, let us introduce for a Polish space $X$ the notation $ \Mp(X) $ denoting the space of finite point measures on $X$ endowed with the topology of weak convergence, and let us represent then the populations $\nu_s(t)$  and $\nu_p(t)$ as finite point measures on the adequate Polish spaces as follows.
 
First, let us consider the population of plants  $\nu_p(t)$ which we simply represent by means of a finite point measure on $\bar{\X}$. That is for $t\ge 0$, the measure is given by 
\begin{align*}
	\nu_p(t) = \sum_{i=1}^{N_p(t)} \delta_{x_i(t)},
\end{align*}
where
\[
N_p(t) := \int_{\bar{\X}} 1 \cdot \nu_p(t)(dx) 
\]
is the total number of plants, and
\begin{align*}
	(x_1(t), \ldots, x_{N_p(t)}(t)) \in \bar{\X}^{N_\alpha(t)}
\end{align*}
corresponds to an arbitrary ordering of the $N_p(t)$ positions (the positions of plants are fixed).

Similarly, we represent the population of seeds $\nu_s(t)$  by means of a finite point measure, but this time on  the product space $\bar{\X} \times \bar{\X}  $:
\begin{align*}
	\nu_s(t) = \sum_{i=1}^{N_s(t)} \delta_{y_i(t), z_i(t)},
\end{align*}
where $N_s(t)$ denotes the total number of seeds at time $t$, and for a  seed $i$ the first coordinate, $y_i(t)$ represents the location of its parent plant (which is fixed in time) and the second coordinate, $z_i(t)$ denotes the location of the seed at time $t$.

In addition to this point measure description, we also require that the evolution of these measures results from the following rules:
\begin{description}
    \item[i)] The groups of seeds generated by a plant are independently generated at exponentially distributed random times, and the number of seeds in a new group is generated according to a counting distribution \mbox{$p\colon \N \cup \lbrace 0 \rbrace \to [0, 1]$}.
    \item[ii)]  The barycenters of the groups at time zero after release are set according to a dispersal kernel $D\colon \bar{\X} \times \bar{\X} \to \R$. Barycenters work as the initial position in space of a newly generated group of seeds. Here we can think that for a given position $x \in \bar{\X}$, the deposit location of the barycenter is drawn from a probability distribution with density $D(x,y)$ with respect to some measure $\bar{D}(dy)$ on $\bar{\X}$. 
    \item[iii)] Within each group, seeds follow independent It\^o's diffusions with infinitesimal generator $L$ given by:
    \begin{align}\label{Itogen}
    L \phi(y) &= a(y) \cdot \nabla \phi(y) + \frac{\sigma(y)^2}{2} \Delta \phi(y),
    \end{align}
    where $\phi: \bar{\X} \to \R$  is an element of one of the following domains:
    \begin{align}\label{ReflecDomain}
    D_r(L) &= \left\{ \phi \in \C^2(\X):  \nabla\phi(y)\cdot \Vec{n}  = 0, \quad \forall y \in \partial \X \right\}, 
    \end{align}
    or
    \begin{align}\label{ReflecDomain}
    D_k(L) &= \left\{ \phi \in \C^2(\X): \phi(y)  = 0, \quad \forall y \in \partial \X \right\}. 
    \end{align}
    Notice that the first domain $D_r$ encodes diffusion of seeds with normal reflection in the boundary of $\X$, while the second represents \textit{killing} of seeds when touching the boundary. Moreover, we also consider the possibility of $\X$ being the whole $\Rd$, in which case the boundary condition on the domain is omitted. 
    \item[iv)]  Each diffusion, representing the trajectory of a seed currently alive, is independently of each other and stopped at an exponentially distributed random time of a given parameter function $\lambda$ (for example with mean proportional to the distance between the source and the deposit location of the group's barycenter). At the stopping time, the seed is assumed to mature and become a new plant with fixed position being equal to the position of the \textit{mother seed} at the time of maturity.
\end{description}

\subsection*{Infinitesimal generator}
Let us consider the $\Mp(\X) \times \Mp(\X^2)$-valued process 
$$\lbrace \nuvec_t : t \geq 0 \rbrace =\lbrace (\nu_p(t), \nu_s(t) ) : t \geq 0 \rbrace.$$ 
The evolution of this process can be described in terms of the infinitesimal generator:
\begin{align}\label{defgen}
\L F_\phivec (\nuvec) &= \L F_\phivec (\nu_p, \nu_s), \nonumber \\
&= \int_{\bar{\X}^2} \lambda(x,y) \left[F_{\phivec}(\nu_p+\delta_y, \nu_s - \delta_{x,y}) -F_{\phivec}(\nu_p,\nu_s) \right] \nu_s(dx,dy) \nonumber \\
&+ \sum_{\kappa \in \N} q(\kappa)  \int_{\bar{\X}^2}  \, D(x,y) \, \left[F_{\phivec}( \nu_p, \nu_s + \kappa \delta_{x,y}) -F_{\phivec}(\nu_p, \nu_s) \right]\, dy \, \nu_p(dx)  \nonumber \\  
&+\left( \int_{\bar{\X}^2} L_z \phi_{s}(y,z) \nu_s(dy,dz) \right) \, \partial_z  F(\langle \phivec,\nuvec \rangle ) +\left( \int_{\bar{\X}^3} \frac{(\sigma(z))^2}{2} |\nabla_z \phi_{s}(y,z) |^2  \nu_s(dy,dz) \right)\  \partial_z^2 F(\langle \phivec,\nuvec \rangle  ),
\end{align}
where:
\begin{equation}
 F_{\phivec}(\nuvec) := F_{\phivec}(\nu_p,\nu_s) := F(\langle \phi_p, \nu_p \rangle,\langle \phi_p, \nu_p \rangle),  
\end{equation}
for $F \in \C^2(\R^2;\R)$, and a pair of functions $\phivec := \lbrace  \phi_p, \phi_s \rbrace$ being such that:
\begin{itemize}
    \item $\phi_\alpha$ is measurable for $\alpha \in \lbrace p, s\rbrace$.
    \item For every $y \in \bar{\X}$, $\phi_{s}(y,\cdot): \bar{\X} \to \R$ is in the domain $D_\iota(L)$ for $\iota \in \lbrace r, k \rbrace$.
\end{itemize}

\begin{remark}
Thereafter, whenever we need to refer to a particular choice of  cylindrical function we will just need to specify the choice of $F$ and the pair $\phivec= \lbrace  \phi_p, \phi_s  \rbrace$. Moreover, for such $\phivec$ we introduce the additional notation
\begin{equation}
    \langle \phivec, \nuvec \rangle := \langle \phi_{p} , \nu_p \rangle_{\bar{\X}}+ \langle \phi_{s} , \nu_s \rangle_{\bar{\X}^2},
\end{equation}
where $\nuvec=(\nu_p,\nu_s)$.
\end{remark}

Each of the terms on the RHS of \eqref{defgen} corresponds to a type of event in the dynamics:
\begin{itemize}
    \item The first term corresponds to maturation of seeds to become plants. This happens at rate $\lambda(x,y)$, where $x \in \X$ is the location of the seed that produced this seed, and $y \in \X$ is the current position of the seed undergoing maturation.  
    \item The second terms corresponds to the release of a group of $\kappa$ new seeds, which are displaced from their mother plant's position $x \in \X$ to a new position $y \in \X$ according to the displacement kernel $D$.
    \item The last two terms correspond to the diffusion of seeds. They come from an application of It\^o's formula to integrals of the finite point measure $\nu_s$. See for example \cite{champagnat_invasion_2007}.
\end{itemize}

Describing the dynamics in terms of the infinitesimal generator \eqref{defgen} is non-rigorous and requires some extra effort to verify that the process it represents is indeed  well-defined. In particular we need the following assumption:

\begin{assumption}\label{AssPherates}
The model parameters are as follows:
\begin{enumerate}
\item The counting distribution $q\colon \N  \to \R$, has finite first and second moments, i.e.:
\begin{equation}
    \mu_1 := \sum_{\kappa\in \N} q(\kappa) \kappa < \infty,   
\end{equation}
and
\begin{equation}
   \sqrt{\mu_2} := \sum_{\kappa\in \N} q(\kappa) \kappa^2 < \infty. 
\end{equation}
\item The dispersal kernel $D(x,y)$ is of compact support in its second argument. Moreover, it is a probability density with respect to some measure $\bar{D}$ on $\bar{\X}$.
\item There exists $\bar{\lambda}>0$ such that the diffusion stopping rate $\lambda \colon \bar{\X} \times \bar{\X} \to \R$ is uniformly bounded:
\begin{equation}
 \sup_{ (x,y) \in \bar{\X}^2}  \lambda(x,y)=: \bar{\lambda} < \infty. \nonumber
\end{equation}
\end{enumerate}
\end{assumption}

We refer to the Appendix for a rigorous description of the processes in terms of Poisson point measures.

\subsection*{Main Results}

Before stating our main result, the following theorem verifies that indeed we have a process which is well-defined.

\begin{theorem}\label{Theoremwelldefined}
Let $\nuvec_0 = (\nu_p(0), \nu_s(0))$ be such that we have:
\begin{equation}\label{pboundt0}
  \mathbb{E} \left(  \langle \mathbf{1}, \nuvec_0 \rangle  \right) < \infty ,    
\end{equation}
where
\begin{equation}
\langle \mathbf{1}, \nuvec_0 \rangle = \int_{ \X} 1 \, \nu_p(0)(dx) +\int_{\X^2} 1 \, \nu_s(0)(dx,dy).
\end{equation}
Then, under Assumption \ref{AssPherates}, part 1), the process $\left(\nuvec_t\right)_{t \geq 0} =\left( \nu_p(t),\nu_s(t) : t \geq 0  \right)$ satisfies:
\begin{equation}\label{pboundtone}
  \mathbb{E} \left(  \sup_{t \in [0,T]}\langle \mathbf{1}, \nuvec_t \rangle  \right) < \infty.     
\end{equation}
In particular, we also have that the process $\nuvec_t$ is well-defined.
\end{theorem}
For a proof we refer to the Appendix, where we proved \eqref{pboundtone} for higher-order moments, and where we also showed that the dynamics of the process $( \nuvec_t)_{t \geq 0}$ indeed corresponds to the one described by the infinitesimal generator $\L$.

Our first main result concerns the derivation of scaling limits for the joint populations of plants and seeds. In order to derive this result, we introduce a rescaling parameter $K$ which will tend to infinity. This  parameter represents the order of the size of the populations and it is used to rescale our populations as follows:
\begin{equation}\label{KmeasDef}
\nu_p^K(t) = \frac{1}{K}\sum_{i=1}^{N_p(t)} \delta_{x_i(t)}, \text{ and } \nu_s^K(t) = \frac{1}{K}\sum_{i=1}^{N_s(t)} \delta_{x_i(t), y_i(t)},
\end{equation}
where we consider the processes $\nu_p^K$ and $\nu_s^K$ as taking values on the spaces $\frac{1}{K}\Mp(\bar{\X})$  and $\frac{1}{K} \Mp(\bar{\X}^2)$ respectively. Additionally, we assume that the model parameters also depend on the scaling parameter $K$ as follows:
\begin{equation}
    \lambda_K(x,y) = \frac{1}{K} \lambda(x,y) \quad \text{ and } D_K(x,y) = \frac{1}{K}D(x,y),
\end{equation}
where $\lambda$ and $D$ satisfy Assumption \ref{AssPherates}.\\

The rescaling \eqref{KmeasDef} can be interpreted as saying that the initial populations of plants and seeds are roughly of the same order, namely the order $K$. To make this idea more precise we assume that at time zero a law of large numbers is satisfied:

\begin{assumption}\label{LLNtimezero}
Assume that the sequence of measures $\lbrace \nuvec_0^{(K)} \rbrace_{K \geq 1}=\lbrace (\nu_p^{(K)}(0), \nu_s^{(K)}(0))\rbrace_{K \geq 1} $ converges weakly to a deterministic measure $\xi(0)=(\xi_p(0), \xi_s(0))$.  Moreover, assume that the measures $\xi_p(0)$, and $\xi_s(0)$, have densities $f_0:\X \to \R_+$, and $g_0:\X^2 \to \R_+$, with respect to the Lebesgue measures on $\X$, and $\X^2$, respectively. 
\end{assumption}

Under Assumption \ref{AssPherates} and Assumption \ref{LLNtimezero} we have the following result:

\begin{theorem}\label{MainThm}
Let the sequence of initial measures $\lbrace \nuvec_0^{(K)} \rbrace_{K \geq 1} = \lbrace (\nu_p^{(K)}(0), \nu_s^{(K)}(0))\rbrace_{K \geq 1}$ be such that:
\begin{equation}\label{SupKtimezero}
    \sup_{ K \in \N} \E \left[  \langle 1, \nu_p(0)^{(K)} \rangle^3 + \langle 1, \nu_s(0)^{(K)} \rangle^3  \right] < \infty.
\end{equation}
Then, for all $T>0$, the sequence $\lbrace \nuvec(t)^{(K)} : t \in [0,T] \rbrace_{K \geq 1} = \lbrace (\nu_p^{(K)}(t), \nu_s^{(K)}(t)): t \in [0,T] \rbrace_{K \geq 1}$ converges in law in $\mathbb{D}([0,T], \M_F(\X) \times \M_F(\X^2))$ to a deterministic $\xi_t =\left( f(t,x) \, dx, g(t,x,y) \, dx \, dy \right)$, element of $\mathbb{C}([0,T], \M_F(\X) \times \M_F(\X^2))$,  where $f$ and $g$ are the unique solution of the following system:
\begin{equation}\label{prestrongplants}
    \partial_t f(t,x) = \int_{\bar{\X}} \lambda(z,x) g(t,z,x) \, dz,
\end{equation}
and
\begin{equation}\label{prestrongseeds}
\partial_t g(t,x,y) = L^* g(t,x,y) -  \lambda(x,y) g(t,x,y) + \mu_1 D(x,y) f(t,x),
\end{equation}
with initial conditions
\begin{equation}
    f(0,x) = f_0(x), \quad g(0,x,y) = g_0(x,y), \quad \forall (x, y) \in \X^2, 
\end{equation}
and for all $t \geq 0, x \in \X$ the functions $g(t,x, \cdot)$ satisfy for all $y \in \partial \X$:
\begin{equation}
\nabla_y g(t,x,y)\cdot \vec{n} = 0, 
\end{equation}
with $\vec{n}$ the normal of $\partial\X$ for the reflecting case, and 
\begin{equation}
g(t,x,y) = 0, 
\end{equation}
for the case in which seeds are killed in the boundary.
\end{theorem}

Of interest by itself, the degenerate PDE system obtained falls into the class of ultra-parabolic PDE system. However, it turns out that classical theory do not apply here and an existence theory  has to be established. In the next results, we provide a first existence theorem in the particular situation of a standard diffusion with reflecting boundary case in a bounded domain $\X$, leading to consider the following coupled PDE system:
\begin{align}
&\partial_t f(t,y)= \int_{\X} \lambda(z,y)g(t,z,y)\,dz &\quad \text{for all}\quad& t>0, y\in  \X, \label{gdm-eq:plant}\\
&\partial_t g(t,x,y)= \nabla_y(A\cdot\nabla_y g(t,x,y)) -\lambda(x,y)g(t,x,y) + \mu_1 D(x,y)f(t,x) &\quad \text{for all}\quad& t>0, (x,y)\in \X^2,\label{gdm-eq:graine}\\
&\nabla_yg(t,x,y)\cdot \Vec{n}=0 &\quad \text{for all}\quad& t>0, x\in \X ,y\in \partial \X, \label{gdm-eq:bc-graine} \\
& f(0,x)=f_0(x) &\quad \text{for all}\quad& x\in  \X,\label{gdm-eq:ic-plant}\\
& g(0,x,y)=0 &\quad \text{for all}\quad&  (x,y)\in  \X^2,\label{gdm-eq:ic-graine}
\end{align}
with $A$ an elliptic matrix and $\X$ a bounded domain of $\R^d$.
\begin{theorem}\label{gdm-thm:pde-n}
Let $\X$ be a smooth bounded domain, with at least a  $C^{1,\alpha}$ boundary. Assume further that $\lambda,D \in W^{1,\infty}(\X^2)$, $f_0\in H^1(\X)$ and $A$ is an elliptic matrix.  Then there exists a   $(f,g)$ such that $f\in C^{1}((0,T),H^1(\X))$, $g \in C^{1}((0,T),H^1(\X)\times H^1(\X))$  solution  to \eqref{gdm-eq:plant}--\eqref{gdm-eq:ic-graine}. In addition, if $D,\lambda$ are non negative  functions, this solution is unique in this class of function. 
\end{theorem}
A careful inspection of the proofs shows that our arguments can readily be transposed to Dirichlet system and so we have
\begin{theorem}\label{gdm-thm:pde-d}
Let $\X$ be a smooth bounded domain, with at least a  $C^{1,\alpha}$ boundary. Assume further that $\lambda,D \in W^{1,\infty}(\X^2)$, $f_0\in H^1_0(\X)$, $A$ an elliptic matrix.  Then there exists a   $(f,g)$ such that $f\in C^{1}((0,T),H^1_0(\X))$, $g \in C^{1}((0,T),H_0^1(\X)\times H^1_0(\X))$  solution  to \eqref{gdm-eq:plant}--\eqref{gdm-eq:ic-graine}. In addition, if $D,\lambda$ are positive functions, this solution is unique in this class of function. 
\end{theorem}

\subsection*{Discussion and comments}

Before going to the proofs of all these results, we would like to make some comments and discuss the potentiality of our results.

First we would like to emphasize that although the ultraparabolic system describing the generator  of the GDM measure value process seems quite unusual in the PDE literature and  in view of the structure of the dispersal processes involved, understanding the main properties of this generator makes it by itself an interesting mathematical object to study. In particular, due to its strong degeneracy, even though the system is linear, its analysis is not completely trivial and need a proper care. Our preliminary analysis sheds some light on how to proceed  to tackle such degeneracy in a variational context to obtain existence and uniqueness results. We rely  on a  vanishing viscosity approach to construct a solution, the solution being then  obtained through the study of the singular limits. The main difficulty in this analysis comes from the establishment of proper uniform \textit{ a priori } estimates in order to ensure the compactness of the sequences of approximated solutions. The structure of the system plays an essential role in obtaining such uniform estimates. The uniqueness of the solution is obtained using  rather standard variational arguments combined with comparison approaches. This first step is also very important for designing efficient numerical schemes to simulate  these 2d+1 problems  with a good accuracy.

Next, we would like to emphasize that  as already mentioned, due to the constraints imposed by the  time-discrete description of grouped dispersal processes, only preliminary modeling studies have been carried out in the literature, see for example
\cite{soubeyrand2011patchy,soubeyrand2014nonstationary,soubeyrand2015evolution,soubeyrand2017group}.
To address the questions raised in these studies in a more  thorough manner, we  believe that our results represent a significant  and necessary step forward. Indeed, we provide a simple and flexible way to model continuous time grouped dispersal processes along with  a clear way to derive large population asymptotics that are simpler to simulate. These new modelling tools will enable the investigation of large population dynamics during long periods of time which for some  evolutionary processes is the natural scale. 
Furthermore, we believe that  these new modelling tools will be of great help to investigate new  questions such as:
 How efficient/evolutionary advantageous a  grouped dispersal strategy is in the context of limited resources in the host environment inducing competition between individuals? What are the properties (e.g., vanishing-time distribution, population-size (quasi-)distribution, spatial extent of the population or in other words distance between the origin of the population and the furthest individual) of a population adopting grouped dispersal in non-homogeneous environment, in particular when the environment is fragmented ? What are population extinction/recolonization rates in metapopulation with inter-population migration driven by grouped dispersal? How the answers to these questions are modulated in the presence of long-distance dispersal?

We leave all these aspects (theoretical and applicative) to further studies.

\subsection*{Organization of the paper}
The rest of our paper is organized as follows. In Section \ref{ExamplesSection} we provide a specific example in order to make transparent the connections of this setting with the one introduced earlier in \cite{soubeyrand2011patchy}. Moreover, in the same section by means of computer simulations we show that our modeling framework can indeed produce patchy patterns for various common choices of dispersal kernels. In the same section we also present how our modelling framework can be used to obtain analytical expressions for observables such as the total number of plants at a given time. In Section \ref{ProofSection} we prove our main theorems using the standard compactness-uniqueness approach. There, we also include the derivation of some standard martingale properties that are used in our proofs. 
\JC{In Section \ref{PDE}, we provide a preliminary analysis of the degenerate PDE ultra parabolic system resulting of the large population  asymptotic of the GDM process and prove the existence results.   
Finally, in the Appendix we include the rigorous definition of the processes on path-space and the proof of some of the necessary propositions in order to rigorously construct a solution to the PDE system.}

\section{A measure valued adaptation of the GDM}\label{ExamplesSection}
As an example of the class of processes we have just introduced, we consider the setting given in Section 4.1 in \cite{soubeyrand2011patchy} adapted to our framework. This means:
\begin{itemize}
    \item The spatial domain is $\ds{ \bar{\X}= [-100,100] \times [-100,100].} $
    \item The counting distribution of the number of particles is given by a negative binomial distribution with mean $\mu_1$ and size parameter $s=\mu_1^2/(\mu_2-\mu_1)$, i.e.,
    \begin{equation}
        q(\kappa) = \frac{\Gamma(\kappa+ s)}{\Gamma(s) \Gamma(\kappa+1)}\left(\frac{\mu_1}{\mu_2}\right)^s \left(1-\frac{\mu_1}{\mu_2}\right)^\kappa. 
    \end{equation}
    \end{itemize}
\begin{itemize}
    \item The dispersal kernel $D: \bar{\X} \times \bar{\X} \to \R$ is of exponential form, i.e.,  given by\\
   
 \[D(x,y)= \frac{1}{2 \pi \beta^2} e^{-\frac{\norm{x-y}}{\beta}}.\]
    \item At time zero the initial populations of plants and seeds are given by
    $\ds{\nu_p(0) = \delta_{0}, \qquad \text{and} \qquad \nu_s(0) = 0.}$
    \end{itemize}   
Moreover, we will assume the seed's maturity rate (i.e. diffusion stopping time's rate) to be 
\begin{equation}\label{lambdaexample}
 \lambda(x,y) = \lambda_0 \norm{y-x},
\end{equation}
for some $\lambda_0 \geq 0$.

\subsection{Simulations}\label{SimSect}
In this section we show some simulations of our model. Let us first introduce some notation to present an algorithm to simulate this process. The notation $\Exp(u)$ denotes an exponentially distributed random variable of parameter $u$. For the measures $\nu_p(t)$, the notation $\Vec{\nu}_p(t)$ denotes an ordered array with the positions of each plant. Analogously for the measure $\nu_s(t)$. The command RandomChoice($A$,Weights$=B$) selects a random element from $A$ according to the weights given by $B$. For both the sake of generality and shortness of exposition, the command Update diffusion($\nu_s, dt$) does as its names suggests, updates the position of the population of seeds according to the relevant It\^o diffusion with the specified time step. 

\begin{algorithm}[H]
  $T_0 \gets 0$\;
  $\nu_p(T_0) \gets \delta_0$\;
  $N_p(T_0) \gets 1$\;
  $\nu_s(T_0) \gets \emptyset$\;
  $N_s(T_0) \gets 0$\;
  $k \gets 0$\;
  %$\text{Permutate}(V,\text{root})$\;
  \While{$T_k < T_{max}$}
  {
    \For{$i\gets0$ \KwTo $N_s(T_k)$}
    {
    $R_{i,k} \gets \int \lambda(x,y) \nu_s(T_k)(dx,dy)$
    }
    $R_k \gets N_p(T_k) +\sum_{i} R_{i,k}$\;
    $\epsilon_k \gets \Simulate(\Exp(R_k))$\;
    $T_{k+1} \gets T_k + \epsilon_k$\;
    $\theta_k \gets \Simulate(\Unif([0,1]))$\;
       \eIf{ $0 \leq \theta_k \leq \frac{N_p(T_k)}{R_k}$}
       {
        $p \gets \Randint(\nu_p(T_k), \text{Weights}=[1, \ldots, 1])$\;
        $n \gets \Simulate(\Progeny)$\;
        $y \gets \Simulate(D(\Vec{\nu}_p(T_k)(p))$\;
        $\nu_s(T_{k}) \gets \nu_s(T_{k}) + n \cdot \delta_{\nu_p(T_k)(p),y}$\;
        $N_s(T_{k+1}) \gets N_s(T_{k}) + n $\;
       }
      {
        $s \gets \Randint(\nu_s(T_k), \text{Weights}=[R_{i,k}])$\;
        $(x,y) \gets \Vec{\nu}_s(T_k)(s)$\;
        $\nu_s(T_{k}) \gets \nu_s(T_{k})- \delta_{x,y}$\;
        $\nu_p(T_{k}) \gets \nu_p(T_{k})+ \delta_{y}$\;
        $N_s(T_{k+1}) \gets N_s(T_{k}) -1 $\;
        $N_p(T_{k+1}) \gets N_p(T_{k}) +1 $\;
      }
      $\nu_p(T_{k+1}) \gets \nu_p(T_{k})$\;
      $\nu_s(T_{k+1}) \gets \text{Update diffusion}(\nu_s(T_k), \epsilon_k)$\;
      $k \gets k+1$\;
  }
  \caption{Simulating the measure valued group dispersal model}
\end{algorithm}

\vspace{0.25cm}

The simulations that we present in this section differ in the choice of the dispersal kernel. Here we focus on the two dimensional case, and use the following dispersal kernels: Gaussian, exponential, and inverse power law. Simulations are run with parameters given by Table \ref{tableparam} below and an It\^o diffusion without drift and $\sigma^2=5$. The parameters for the dispersal kernels are pick such that an average displacement of 10 units is enforced. 

\begin{table}[H]
\centering
\begin{tabular}{|c|c|}
\hline
Parameter      & Value                                                                                                                              \\ \hline
$\bar{\X}$     & $[-100,100] \times [-100,100]$                                                                                                     \\ \hline
$q(\kappa)$    & $\frac{\Gamma(\kappa+ s)}{\Gamma(s) \Gamma(\kappa+1)}\left(\frac{\mu_1}{\mu_2}\right)^s \left(1-\frac{\mu_1}{\mu_2}\right)^\kappa$ \\ \hline
$\mu_i$        & $\mu_1=1$ and $\mu_2=25$                                                                                                            \\ \hline
$\lambda(x,y)$ & $0.05\cdot \norm{x}$                                                                                                                                 \\ \hline
$\nu_p(0)$     & $\delta_0$                                                                                                                         \\ \hline
$\nu_s(0)$     & 0                                                                                                                                  \\ \hline
\end{tabular}
\caption{Parameters used in simulations. The dispersal kernel is different in each simulation case.}
\label{tableparam}
\end{table}

For the three cases of dispersal kernel mentioned above, we show a run of the simulations on which patchy patterns emerge. All simulations are stopped until a total population of 2000 plants is reached. After this, a smoothing is performed using the Gaussian kde method from the scipy library of Python3.

\subsubsection*{Exponential kernel}
In this case the kernel is given by:
\begin{equation}
   D(x,y)= \frac{1}{2 \pi \beta^2} e^{-\frac{\norm{x-y}}{\beta}} 
\end{equation}
for some $\beta >0$.

\begin{figure}[H]
    \centering
    \includegraphics[scale=0.6]{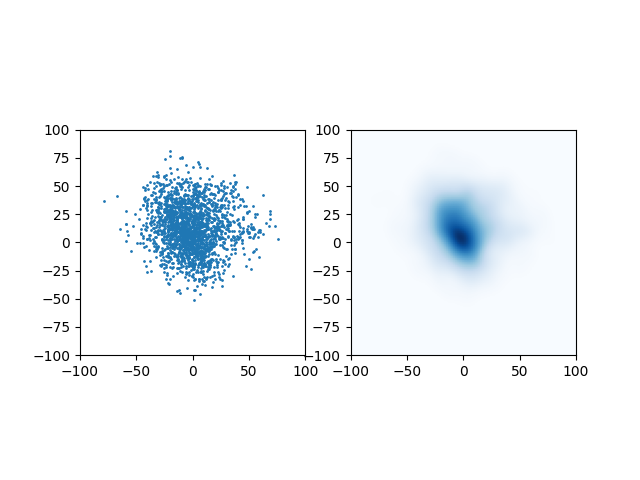}
    \caption{ Patchy pattern observed in a simulation using an exponential dispersal kernel with mean displacement equal to 10, and where seeds follow diffusion with killing in the boundary. Left: pattern of particles. Right: estimated point intensity obtained by kernel smoothing with a Gaussian kernel.}
    \label{fig:Exponential}
\end{figure}

\begin{figure}[H]
    \centering
    \includegraphics[scale=0.6]{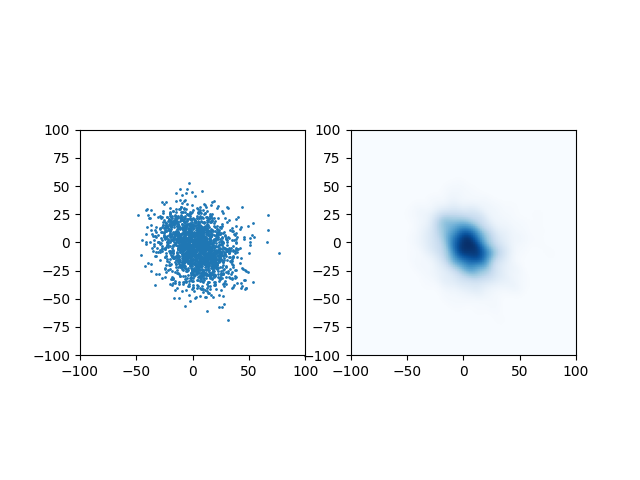}
    \caption{ Patchy pattern observed in a simulation using an exponential dispersal kernel with mean displacement equal to 10, and where seeds follow diffusion with reflection in the boundary. Left: pattern of particles. Right: estimated point intensity obtained by kernel smoothing with a Gaussian kernel.}
    \label{fig:Exponential}
\end{figure}

\subsubsection*{Gaussian kernel}
The Gaussian kernel case is given by:
\begin{equation}
D(x,y)= \frac{1}{2 \pi \beta^2} e^{-\frac{\norm{x-y}^2}{\beta^2}} 
\end{equation}
for $\beta >0$.

\begin{figure}[H]
    \centering
    \includegraphics[keepaspectratio=true,scale=0.6]{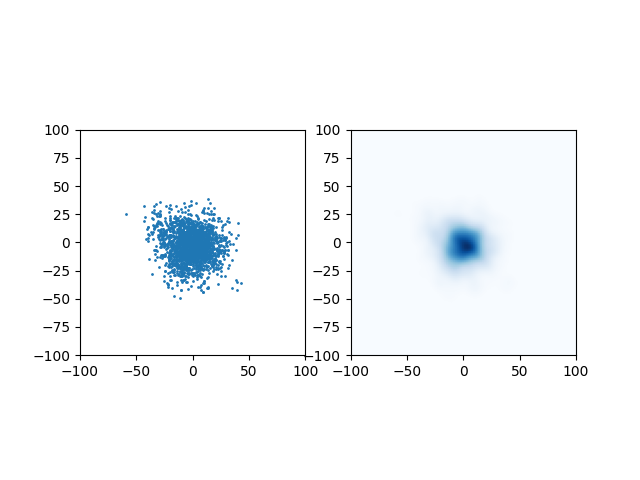}
    \caption{Patchy pattern observed in a simulation using a Gaussian dispersal kernel with mean displacement 10, and where seeds follow diffusion with killing in the boundary. Left: pattern of particles. Right: estimated point intensity obtained by kernel smoothing with a Gaussian kernel.}
    \label{fig:Gaussian}
\end{figure}

\begin{figure}[H]
    \centering
    \includegraphics[keepaspectratio=true,scale=0.6]{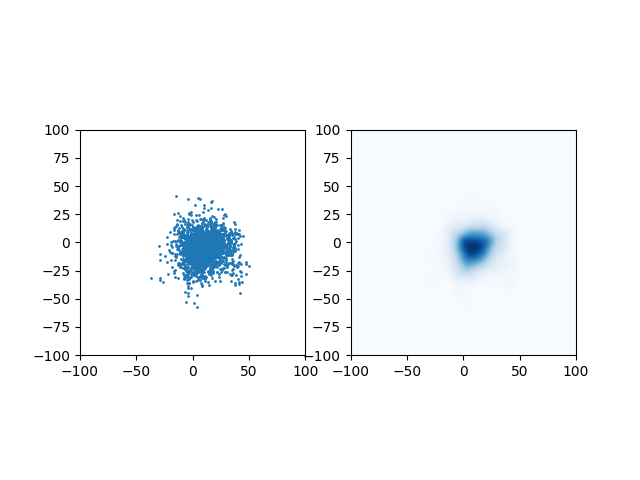}
    \caption{Patchy pattern observed in a simulation using a Gaussian dispersal kernel with mean displacement 10, and where seeds follow diffusion with reflection in the boundary. Left: pattern of particles. Right: estimated point intensity obtained by kernel smoothing with a Gaussian kernel.}
    \label{fig:Gaussian}
\end{figure}

\subsubsection*{Power-law kernel}
For $\beta >0$ and $a>2$, the power-law kernel takes the form.
\begin{equation}
D(x,y)=\frac{(a-2)(a-1)}{2 \pi \beta^2} (1+\frac{\norm{x-y}}{\beta})^{-a}  
\end{equation}

\begin{figure}[H]
    \centering
    \includegraphics[keepaspectratio=true,scale=0.6]{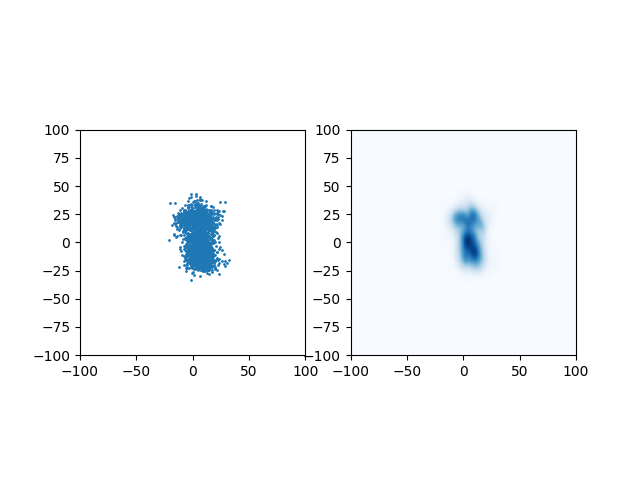}
    \caption{Patchy pattern observed in a simulation using a power-law dispersal kernel of parameters $\beta=5$, and $a=4$ (i.e. with mean displacement 10), and where seeds follow diffusion with killing in the boundary. Left: pattern of particles. Right: estimated point intensity obtained by kernel smoothing with a Gaussian kernel.}
    \label{fig:powerlaw}
\end{figure}

\begin{figure}[H]
    \centering
    \includegraphics[keepaspectratio=true,scale=0.6]{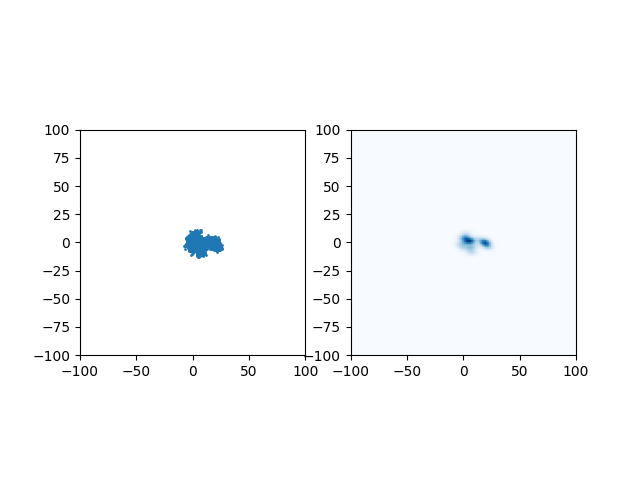}
    \caption{Patchy pattern observed in a simulation using a power-law dispersal kernel of parameters $\beta=5$, and $a=4$ (i.e. with mean displacement 10), and where seeds follow diffusion with reflection in the boundary. Left: pattern of particles. Right: estimated point intensity obtained by kernel smoothing with a Gaussian kernel.}
    \label{fig:powerlaw}
\end{figure}

\subsection{Population level descriptors as observables}
Here, we focus on the simple situation in which seeds mature at constant rate, i.e., $\lambda(x,y) = 1$ for all $x,y \in \bar{\X}$.  Let us denote by $N_p(t)$ and $N_s(t)$ the random number of plants and seeds at time $t$, respectively. These two important quantities can be obtained by integration of the measures $\nu_p(t)$ and $\nu_s(t)$ against the constant function one:
\begin{equation}
    N_p(t) = \langle 1, \nu_p(t) \rangle, \text{ and } N_s(t) = \langle 1, \nu_s(t) \rangle,
\end{equation}
where we have abused notation by representing with same symbol, $\langle \cdot, \cdot \rangle$, integration over the spaces $\X$ and $\X^2$.\\

Using the algorithm of Section \ref{SimSect}, we obtain the following quantitative scenario for the expected number of plants:

\begin{figure}[H]
    \centering
    \includegraphics[keepaspectratio=true,scale=0.45]{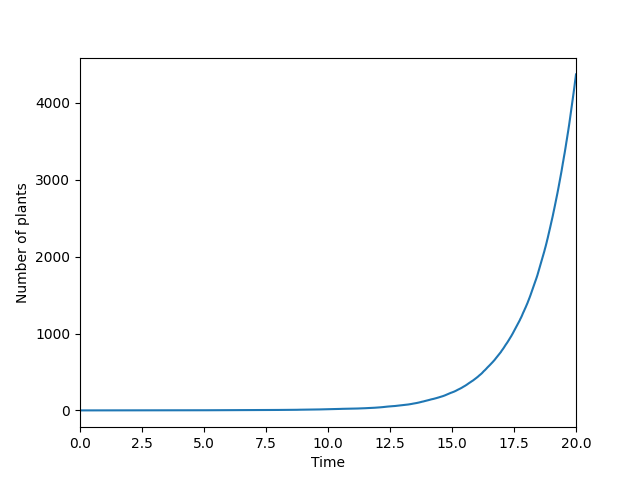}
    \caption{Expected number of plants with a time horizon $T=20$. Simulations were run with an exponential dispersal kernel (truncated to $[-100,100] \times [-100,100]$), and a constant maturity rate $\lambda=1$.}
    \label{fig:my_label}
\end{figure}

Thanks to the knowledge of the generator $\L$ have the following result concerning the expectation of $N_p(t)$ and $N_s(t)$ when starting from a single plant at the origin.

\begin{proposition}\label{ExpNumPlant}
Let $\nuvec(0)=(\delta_0, 0)$, then we have:
\begin{align}
    \E N_p(t) &= \frac{ \left( (\sqrt{4\mu+1}-1)e^{\frac{1}{2}(\sqrt{4\mu+1}-1)t} +(\sqrt{4\mu+1}  +1)e^{-\frac{1}{2}(\sqrt{4\mu+1}+1)t} \right)}{2\sqrt{4\mu+1}} ,\\
    \E N_s(t) &= \frac{ \mu \left( e^{\frac{1}{2}(\sqrt{4\mu+1}-1)t} -e^{-\frac{1}{2}(\sqrt{4\mu+1}+1)t} \right)}{\sqrt{4\mu+1}} .
\end{align}
where $\mu$ is given by:
\begin{equation}
    \mu = \sum_{ \kappa} q(\kappa) \kappa.
\end{equation}
\end{proposition}

\section{Proofs}\label{ProofSection}

We will not show Theorem \ref{MainThm} directly, instead we are going to provide a formulation of it in IDE form, i.e. in the form of Theorem \ref{MainThmGeneral} below.

\subsection{Scaling limits: general statement}
We first relax Assumption \ref{LLNtimezero} by drooping the required existence of densities at time zero.

\begin{assumption}\label{LLNtimezeroprime}
Assume that the sequence of measures $\lbrace \nuvec_0^{(K)} \rbrace_{K \geq 1}=\lbrace (\nu_p^{(K)}(0), \nu_s^{(K)}(0))\rbrace_{K \geq 1} $ converges weakly to a deterministic measure $\xi(0)=(\xi_p(0), \xi_s(0))$.
\end{assumption}

Under Assumption \ref{AssPherates} and Assumption \ref{LLNtimezeroprime} we have the following result:

\begin{theorem}\label{MainThmGeneral}
Let the sequence of initial measures $\lbrace \nuvec_0^{(K)} \rbrace_{K \geq 1} = \lbrace (\nu_p^{(K)}(0), \nu_s^{(K)}(0))\rbrace_{K \geq 1}$ be such that for all $K \in \N$ we have:
\begin{equation}
    \E \left[  \langle 1, \nu_p(0)^{(K)} \rangle^3 + \langle 1, \nu_s(0)^{(K)} \rangle^3  \right] < \infty.
\end{equation}
Then, for all $T>0$, the sequence $\lbrace \nuvec(t)^{(K)} : t \in [0,T] \rbrace_{K \geq 1} = \lbrace (\nu_p^{(K)}(t), \nu_s^{(K)}(t)): t \in [0,T] \rbrace_{K \geq 1}$ converges in law in $\mathbb{D}([0,T], \M_F(\X) \times \M_F(\X^2))$ to a deterministic continuous function $\xi_t =( \xi_p(t), \xi_s(t))$ being the unique solution of the following system:
\begin{align}\label{PlantIDE}
  \langle \xi_p(t) , \phi_p\rangle_{\bar{\X}}- \langle \xi_p(0) , \phi_p\rangle_{\bar{\X}} = \int_0^t \int_{\bar{\X}^2} \lambda(x,y) \phi_p(y) \xi_s(s)(dx,dy) \,  ds 
\end{align}
\begin{align}\label{SeedIDE}
    \langle \xi_s(t) , \phi_s\rangle_{\bar{\X}^2} - \langle \xi_s(0) , \phi_s\rangle_{\bar{\X}^2} &= \int_0^t \langle \xi_s(t) , \L_y \phi_s(x,y) \rangle_{\bar{\X}^2} \, ds -\int_0^t \int_{\bar{\X}^2} \lambda(x,y) \phi_s(x,y) \xi_s(s)(dx,dy) \,  ds  \nonumber \\
    & + \mu_1  \int_0^t \int_{\bar{\X}^2} D(x,y) \phi_s(x,y) \xi_p(s)(dx) \, dy \,  ds
\end{align}
with initial conditions
\begin{equation}
    \xi_p(0)= \xi \text{ and } \xi_s(0) = \xi, 
\end{equation}
and for all $x$ with $\phi_s(x,\cdot)$ being in the domain of the diffusive generator $L$.
\end{theorem}

\subsection{Martingale characterization}

In this section we discuss the proofs of our main results. Let us start with some well-known martingales associated to Markov processes. We start with a simple application of Dynkin's theorem:

\begin{theorem}\label{DynkinMarProp}
Let $\nuvec_0 = (\nu_p(0), \nu_s(0))$ be such that for some $r \geq 2$ we have:
\begin{equation}\label{pboundt0Dynkin}
  \mathbb{E} \left(  \langle 1, \nu_p(0) \rangle^r + \langle 1, \nu_s(0) \rangle^r  \right) < \infty,
\end{equation}
and 
\begin{equation}
\mu_r := \sum_{\kappa} q(\kappa) \kappa^r < \infty.
\end{equation} 
Let also $F$ and  $\phivec$, be such that  there exists a $C$, possibly dependant on $F$, and $\phivec$, such that for all $\nuvec \in \M_p$:
\begin{equation}\label{ControlFplusgen}
|F_{\phivec} (\nuvec)| +  |\L F_{\phivec} (\nuvec)|  \leq C \,\left( 1+ \langle \onevec, \nuvec \rangle^r \right). 
\end{equation}
Then, under Assumption \ref{AssPherates}, we have that the process 
\begin{equation}\label{DynkinMart}
M_t(F_{\phivec}) = F_{\phivec} (\nuvec(t)) - F_{\phivec} (\nuvec(0))- \int_0^t \L F_{\phivec} (\nuvec(s)) \, ds
\end{equation}
is a càdlàg martingale. 
\end{theorem}

\begin{proof}
From Proposition \ref{PropGenerator} and Dynkin's theorem we know that $M_t(F_{\phivec})$ is a local martingale. Hence, it is enough to show that the R.H.S. of \eqref{DynkinMart} is integrable. This is a consequence of assumption  \eqref{ControlFplusgen} and Proposition \ref{Propcontrolp}.
\end{proof}

\begin{remark}\label{UsingDynkinMart}
Simple but tedious computations show that for, $1 \leq  m \leq r-1$,   the functions
\begin{align}
  F_{p}^m(\nu)  =  \left( \langle \phi_p , \nu_p \rangle_{\bar{\X}} \right)^m, \text{ and } F_{s}^m(\nu)  =  \left( \langle \phi_s , \nu_s \rangle_{\bar{\X}} \right)^m
\end{align}
satisfy condition \eqref{ControlFplusgen}. 
\end{remark}

The following result is a consequence of Proposition \ref{DynkinMarProp} and  Remark \ref{UsingDynkinMart}

\begin{proposition}\label{MartingalesProp}
Let $\phi_p$ and $\phi_s$ be as in Remark \ref{UsingDynkinMart}, then under Assumption \ref{AssPherates},
we have the following cádlág martingales:
\begin{enumerate}
\item For the population of plants:
\begin{align}\label{MartH}
    M_t^{\phi_p}  &= \langle \nu_p(t) , \phi_p\rangle_{\bar{\X}}- \langle \nu_p(0) , \phi_p\rangle_{\bar{\X}} - \int_0^t \int_{\bar{\X}^2} \lambda(x,y) \phi_p(y) \nu_s(s)(dx,dy) \,  ds 
\end{align}
is a cádlág martingale with predictable quadratic variation
\begin{align}\label{MartHQV}
 \langle M^{\phi_v} \rangle_t &= \int_0^t \int_{\bar{\X}^2} \lambda(x,y) \, \phi_p(y)^2 \, \nu_s(s)(dx,dz) \,  ds
\end{align}
\item For the population of seeds:
\begin{align}\label{MartF}
    M_t^{\phi_s}  &= \langle \nu_s(t) , \phi_s\rangle_{\bar{\X}^2} - \langle \nu_s(0) , \phi_s\rangle_{\bar{\X}^2} -\int_0^t \langle \nu_s(t) , \L_y \phi_s(x,y) \rangle_{\bar{\X}^2} \, ds \nonumber \\
    &+\int_0^t \int_{\bar{\X}^2} \lambda(x,y) \phi_s(x,y) \nu_s(s)(dx,dy) \,  ds  \nonumber \\
    & - \sum_{ \kappa} q\kappa) \kappa \int_0^t \int_{\bar{\X}^2} D(x,y) \phi_s(x,y) \nu_p(s)(dx) \, dy \,  ds
\end{align}
is a cádlág martingale with predictable  quadratic variation
\begin{align}\label{MartFQV}
 \langle M^{\phi_s}  \rangle_t &=  \int_0^t \int_{\bar{\X}^2} \sigma^S(y)^2  \, |\nabla_y \phi_s(x,y) |^2 \, \nu_s(s)(dx,dy) \, ds +\int_0^t \int_{\bar{\X}^2} \lambda(x,y) \, \phi_s(x,y)^2 \, \nu_s(s)(dx,dz) \,  ds \nonumber \\
 &+ \sum_{ \kappa} q(\kappa) \kappa^2 \int_0^t \int_{\bar{\X}^2} D(x,y) \phi_s(x,y)^2 \nu_p(s)(dx) \, dy \,  ds
\end{align}
\end{enumerate}
\end{proposition}

\begin{proof}
Assume $r\geq 2$. By Proposition \ref{DynkinMarProp} and Remark \ref{UsingDynkinMart} with $m=1$, we have that $M_t^{\phi_p}$ and $M_t^{\phi_s}$ are cádlág martingales. The predictable quadratic variation is also obtained by standard arguments. However, we show the simple case $\langle M^{\phi_p} \rangle_t$ for completeness of exposition.\\

Using Remark \ref{UsingDynkinMart}, with $q=2$, we obtain that the following is a cádlág martingale:
\begin{align}
\langle \phi_p, \nu_p(t) \rangle^2 - \langle \phi_p, \nu_p(0) \rangle^2 - \int_0^t \int_{\bar{\X}^2} \lambda(x,y) \left[ \left(\langle \phi_p, \nu_p(s) \rangle + \phi_p(y) \right)^2 - \langle \phi_p, \nu_p(s) \rangle^2  \right] \, \nu_s(s)(dx,dy)   \, ds,
\end{align}
On the other hand, using Ito\^s formula on \eqref{MartH} gives the additional martingale:
\begin{align}
\langle \phi_p, \nu_p(t) \rangle^2 - \langle \phi_p, \nu_p(0) \rangle^2 - 2 \int_0^t \langle \phi_p, \nu_p(s) \rangle  \int_{\bar{\X}^2} \lambda(x,y) \, \phi_p(y) \, \nu_s(s)(dx,dy)   \, ds - \langle M^{\phi_p} \rangle_t.
\end{align}
We conclude using the uniqueness of the Doob-Meyer's decomposition.
\end{proof}
We now present some ways on how to use Proposition and Proposition \ref{MartingalesProp} to deduce some properties of our process. Let us start with the latter.

\subsection{Proof: expected number of plants}
Thanks to the Martingale characterization, the proof of Proposition \ref{ExpNumPlant} becomes easily accessible.

\begin{proof}[Proof of Proposition \ref{ExpNumPlant}]
We can use the martingale characterization given in Proposition \ref{MartingalesProp}, with \mbox{$\phi_p(\cdot )=1=\phi_s(\cdot ,\cdot)$}, to obtain the system
\begin{align}
    \E N_p(t) &= 1+ \int_0^t \E N_s(s) \, ds,  \\
    \E N_s(t) &= \mu \cdot  \int_0^t \E N_p(s) \, ds - \int_0^t \E N_s(s) \, ds,
\end{align}
with initial conditions $N_p(0)=1$ and $N_s(0)=0$.  Solving this first order linear system finishes the proof.
\end{proof}

\subsection{Proof of Theorem \ref{MainThm}}

Here we show how Theorem \ref{MainThm} is an application of Theorem \ref{MainThmGeneral} under the additional assumption of having densities at time zero. In this section we additionally assume the following:
\begin{itemize}
    \item The spatial diffusion to be given by standard reflected Brownian motion. This means zero drift, and $\sigma=1$.
    \item The existence of densities at time zero, and the propagation of absolutely continuity for later times.
\end{itemize}
Let us first start with the equations concerning the population of plants. Under the above assumptions we can rewrite \eqref{PlantIDE} as follows:
\begin{align}\label{PlantweakIDE}
 \int_{\bar{\X}}  \phi_p(x) f(t,x) \, dx-\int_{\bar{\X}}  \phi_p(x) f(0,x) \, dx = \int_0^t \int_{\bar{\X}^2} \lambda(z,x) \,  \phi_p(x) \, g(s,z,x) \, dz \, dx \,  ds,
\end{align}
where $f(t,x)$ denotes the density of plants at position $x \in \bar{\X}$, at time $t$.\\

Differentiating with respect to time we obtain
\begin{align}\label{PlantweakIDE2}
 \int_{\bar{\X}}  \phi_p(x) \partial_t f(t,x) \, dx= \int_{\bar{\X}^2} \lambda(z,x) \,  \phi_p(x) \, g(t,z,x) \, dz \, dx.
\end{align}
By DuBois-Reymond lemma, we have the strong form
\begin{equation}\label{prestrongplants}
    \partial_t f(t,x) = \int_{\bar{\X}} \lambda(z,x) g(t,z,x) \, dz.
\end{equation}
Analogously for \eqref{SeedIDE} we have the strong formulation
\begin{equation}\label{prestrongseeds}
\partial_t g(t,x,y) = \frac{1}{2}\Delta_y g(t,x,y) -  \lambda(x,y) g(t,x,y) + \mu_1 D(x,y) f(t,x).
\end{equation}

\subsection{Proof of Theorem \ref{MainThmGeneral}}
For the proof of Theorem \ref{MainThmGeneral} we are going to use the standard compactness-uniqueness approach:

\subsection*{Step 1: show uniqueness of solutions}
Let us first assume that $(\xi_p(t),\xi_s(t))_{t \geq 0}$, and $(\bar{\xi}_p(t),\bar{\xi}_s(t))_{t \geq 0}$ are solutions of \eqref{PlantIDE}-\eqref{SeedIDE}. We want to show that for $\alpha \in \lbrace p,s \rbrace$ we have:
\begin{equation}
    \norm{\xi_\alpha - \bar{\xi}_\alpha}_{\alpha} = 0,
\end{equation}
where for $\nu_\alpha^1$ and $\nu_\alpha^2$ their variation norm is given by
\begin{equation}
    \norm{\nu_\alpha^1 - \nu_\alpha^2} = \sup_{\substack{\phi_\alpha \in L^\infty \\ \norm{\phi_\alpha}_\infty\leq1}} \lvert \langle \nu_\alpha^1 - \nu_\alpha^2 , \phi_\alpha \rangle \rvert, 
\end{equation}
where the inner products are taken in the appropriate spaces.\\

Let us first deal with the case $\alpha= p$. Let $\phi_p$ be such that $\norm{\phi_p}_\infty\leq 1$, by \eqref{PlantIDE} and Assumption \ref{AssPherates}, we have
\begin{align}\label{uniquennxiv}
\left\lvert \langle \xi_p(t) -\bar{\xi}_p(t) , \phi_p\rangle \right\rvert &\leq \bar{\lambda} \int_0^t \sup_{\phi \leq 1} \left\lvert \langle \xi_s(t) -\bar{\xi}_s(t) , \phi \rangle \right\rvert \,  ds.
\end{align}
For the case $\alpha=s$, we consider the semi-groups $P_s(t)$ corresponding to the diffusion process with generator $L$. Let us fix a function $\phi_s \in L^\infty(\X^2)$ with $\norm{\phi_s } \leq 1$, take $t \in [0,T]$, and define the following function:
\begin{equation}\label{timezationphi}
  \phi_s(s,x,y) = P_s (t-s) \phi_s(x,y) \quad \forall (x,y) \in \X^2, 
\end{equation}
where the semigroup acts on the second variable of the original $\phi_s$.\\

By construction $\phi_s(s,x,y)$ is a solution of the boundary problem:
\begin{align}\label{boundprobuniq}
    \partial_s \phi_s(s, x,y) &+ L \phi_s(s, x,y) = 0 \quad \text{ on } [0,T] \times \X^2,  \nonumber \\
    \nabla_y \phi_s(s,x, y)\cdot \Vec{n} &= 0 \quad \text{ on } [0,T] \times \X \times  \partial \X, \\
    \lim_{ s \to t} \phi_s(s,x, y) &=  \phi_s( x,y) \quad \text{ on } \X^2. \nonumber
\end{align}
The weak time-space formulation of \eqref{SeedIDE}, is given by
\begin{align}\label{Seedunique}
 \langle \xi_s(t) , \phi_s(t, \cdot) \rangle &-  \langle \xi_s(0) , \phi_s(0,\cdot) \rangle  =  \int_0^t \int_{\X^2}  (\L_y + \partial_s )\phi_s(s,x,y) \nu_s(t)(dx,dy) \, ds \nonumber \\
 &-\int_0^t \int_{\bar{\X}^2} \lambda(x,y) \phi_s(s,x,y) \nu_s(s)(dx,dy) \,  ds  + \mu_1  \int_0^t \int_{\bar{\X}^2} D(x,y) \phi_s(s,x,y) \nu_p(s)(dx) \, dy \,  ds,
\end{align}
and hence 
\begin{align}\label{Seedunique}
\left\lvert \langle \xi_s(t) -\bar{\xi}_s(t) , \phi_s(x,y)\rangle \right\rvert &\leq \bar{\lambda} \int_0^t \sup_{\phi \leq 1} \left\lvert \langle \xi_s(t) -\bar{\xi}_s(t) ,  \phi \rangle \right\rvert \,  ds + \mu_1  \int_0^t \sup_{\phi \leq 1} \left\lvert \langle \xi_p(t) -\bar{\xi}_p(t) ,  \phi \rangle \right\rvert \ \,  ds,
\end{align}
where we used 
\begin{equation*}
 \sup_{\phi \leq 1} \left\lvert \langle \xi_p(t) -\bar{\xi}_p(t) ,  P_s(t-s) \phi \rangle \right\rvert \leq \sup_{\phi \leq 1} \left\lvert \langle \xi_p(t) -\bar{\xi}_p(t) ,  \phi \rangle \right\rvert,   
\end{equation*}
which is a consequence of the maximal principle applied to \eqref{boundprobuniq}.\\

Taking first the sup on the LHS of each case, and then summation gives the inequality:
\begin{align}
\sum_{\alpha \in \lbrace p, s \rbrace} \, \sup_{\phi \leq 1}  \,  \left\lvert \langle \xi_\alpha(t) -\bar{\xi}_\alpha(t) , \phi_\alpha\rangle \right\rvert &\leq C \left( \int_0^t \sum_{\alpha \in \lbrace p, s \rbrace} \, \sup_{\phi \leq 1}  \,  \lvert \langle \xi_\alpha(s) - \bar{\xi}_\alpha(s), \phi_\alpha \rangle \rvert  \, ds \right), \nn 
\end{align}
and  by Gronwall's lemma we conclude uniqueness.

\subsection*{Step 2: uniform estimates}
Under the rescaling given by \eqref{KmeasDef}, a consequence of Proposition \ref{MartingalesProp} is the following:

\begin{proposition}\label{KMartingalesProp}
We have the following cádlág martingales:
\begin{enumerate}
\item For the population of plants,
\begin{align}\label{MartH}
    M_t^{K,\phi_p}  &= \langle \nu_p^{(K)}(t) , \phi_p\rangle_{\bar{\X}}- \langle \nu_p^{(K)}(0) , \phi_p\rangle_{\bar{\X}} - \int_0^t \int_{\bar{\X}^2} \lambda(x,y) \phi_p(y) \nu_s^{(K)}(s)(dx,dy) \,  ds
\end{align}
is a cádlág martingale with predictable quadratic variation
\begin{align}\label{MartHQV}
 \langle M^{K,\phi_v} \rangle_t &= \frac{1}{K}\int_0^t \int_{\bar{\X}^2} \lambda(x,y) \, \phi_p(y)^2 \, \nu_s^{(K)}(s)(dx,dz) \,  ds.
\end{align}
\item For the population of seeds,
\begin{align}\label{MartF}
    M_t^{K,\phi_s}  &= \langle \nu_s^{(K)}(t) , \phi_s\rangle_{\bar{\X}^2} - \langle \nu_s^{(K)}(0) , \phi_s\rangle_{\bar{\X}^2} -\int_0^t \langle \nu_s^{(K)}(t) , \L_y \phi_s(x,y) \rangle_{\bar{\X}^2} \, ds \nonumber \\
    &+\int_0^t \int_{\bar{\X}^2} \lambda(x,y) \phi_s(x,y) \nu_s^{(K)}(s)(dx,dy) \,  ds  \nonumber \\
    & - \mu_1 \int_0^t \int_{\bar{\X}^2} D(x,y) \phi_s(x,y) \nu_p^{(K)}(s)(dx) \, dy \,  ds
\end{align}
is a cádlág martingale with predictable quadratic variation
\begin{align}\label{MartFQV}
 \langle M^{K,\phi_s}  \rangle_t &=\frac{1}{K}  \int_0^t \int_{\bar{\X}^2} \sigma^S(y)^2  \, |\nabla_y \phi_s(x,y) |^2 \, \nu_s^{(K)}(s)(dx,dy) \, ds +\frac{1}{K}\int_0^t \int_{\bar{\X}^2} \lambda(x,y) \, \phi_s(x,y)^2 \, \nu_s(s)(dx,dz) \,  ds \nonumber \\
 &+\frac{\mu_2}{K}  \int_0^t \int_{\bar{\X}^2} D(x,y) \phi_s(x,y)^2 \nu_p^{(K)}(s)(dx) \, dy \,  ds.
\end{align}
\end{enumerate}
\end{proposition}

Notice that by \eqref{SupKtimezero}, for $T>0$, we have
\begin{equation}\label{unifoverKontrol}
    \sup_{K \in \N} \E \left( \sup_{t \in [0,T]} \langle \nu_\alpha^{(K)}(t),1 \rangle^3 \right) < \infty 
\end{equation}
for all $\alpha \in \lbrace p, s \rbrace$. This in turn implies the uniform estimate:
\begin{equation}\label{uniformphialpha}
    \sup_{K \in \N} \E \left( \sup_{t \in [0,T]} \langle \nu_\alpha^{(K)}(t), \phi_\alpha \rangle^3 \right) < \infty,
\end{equation}
for $\phi_\alpha$ bounded and measurable, and in particular $\phi_s \in D(L)$.

\subsection*{Step 3: tightness}
Let us consider the spaces of finite measures $\M_F^p:=\M_F(\X)$ and $\M_F^s:=\M_F(\X^2)$, both equipped with the vague topology. Here we show that the sequences of laws $Q_\alpha^K$ of the processes $\nu_\alpha^{(K)}$ are uniformly tight in $\P(\D([0,T],\M_F^\alpha))$, for $\alpha \in \lbrace p, s \rbrace$. We then use Theorem 2.1 from \cite{roelly1986criterion}. This means that we need to find dense sets $D_\alpha$, dense in the space of continuous functions over $\X$, and $\X^2$, respectively, such that the real-valued sequences  $\langle \nu_\alpha^{(K)}, f\rangle$ are tight in $\P(\D([0,T],\R)$ for all $f_\alpha$ in $D_\alpha$. For the sets $D_\alpha$, we set them to be $C_b(\X)$, and $D(\L)$, respectively for $\alpha=p$ and $\alpha=s$, respectively. By Aldous and Rebolledo's criteria, it is enough to show:
\begin{equation}\label{uniformphialphatight}
    \sup_{K \in \N} \E \left( \sup_{ t \in [0,T]} |\langle \nu_\alpha(t)^{(K)} , f_\alpha \rangle |  \right) < \infty,
\end{equation}
and the tightness of the quadratic variation and the drift part of the martingales given in Proposition \ref{KMartingalesProp} above. 

Notice that \eqref{uniformphialphatight} is already given in \eqref{uniformphialpha}. The tightness of quadratic variations and the drift part are done the same way. We then only show tightness for the quadratic variation part. Let $\delta >0$, and consider stopping times $\tau, \tau\myprime$ such that
\[
0 \leq \tau \leq \tau\myprime \leq \tau + \delta \leq T.
\]
By Doob's inequality and  the estimate \eqref{unifoverKontrol}  we have
\begin{align}\label{Doobsapp}
\E \left( \langle M_K^{\alpha}  \rangle_{\tau} -   \langle M_K^{\alpha}  \rangle_{\tau\myprime}  \right) \leq C \delta. 
\end{align}
Since similar estimates can be obtain for the drift part of the martingales, we conclude the proof of uniform tightness of the sequence $Q_\alpha^{K}$.

\subsection*{Step 4: characterization of limit points}
Now that we have shown tightness, we can assume convergence of a sub-sequence of the laws $Q_\alpha^K$. Let us denote by $Q_\alpha$ the limiting law of this sub-sequence, that with an abuse of notation we denote again by $Q_\alpha^K$. Let us denote by $\Lambda_\alpha$ a process with law $Q_\alpha$. We need to verify that it is almost surely strongly continuous. We follow step 5, on page 54 of \cite{meleard_stochastic_2015}, with a small modification due to the branching mechanism. For $j \in \N$, let us introduce 
\begin{equation}
    \kappa_j = \inf \lbrace \kappa \in \N: \sum_{i=1}^{\kappa_j} p(i) \geq 1 -\frac{1}{2^j}  \rbrace.
\end{equation}
Notice that by construction of our processes we have
\begin{equation}
\P \left(    \sup_{ t \in [0,T]} \sup_{ f \leq 1} \,  \left\lvert \langle \nu_\alpha^{(K)}(t), f \rangle - \langle \nu_\alpha^{(K)}(t-), f \rangle \right\rvert  \leq \frac{\kappa_j}{K} \right) \geq 1 -\frac{1}{2^j}.
\end{equation}
Similar arguments to those in \cite{meleard_stochastic_2015}, together with the continuity of the mapping $\nu \mapsto \sup_{ t \in [0,T]} \,  \left\lvert \langle \nu(t), f \rangle - \langle \nu(t-), f \rangle \right\rvert $ on path space, verify that indeed $\Lambda_\alpha$ is a.s. strongly continuous. The rest of the characterization is a standard use of the martingales given in \eqref{KMartingalesProp}, and follows the lines of \cite{fournier_microscopic_2004,meleard_stochastic_2015,champagnat_invasion_2007}.

\section{Properties of the PDE system}\label{PDE}

Let us now look a bit more at the properties of the PDE system satisfied by the population of seeds and  plants, respectively $g$ and $f$.To have a tractable analysis, we restrict our analysis to bounded domain $\X$ and we will present our argument with  $A(x)=Id$.  All the arguments we use can be applied to a general elliptic matrix at the price of more complex notations.   
In this simplified situation,  $f$ and $g$ satisfy

\begin{align}
&\partial_t f(t,y)= \int_{\X} \lambda(z,y)g(t,z,y)\,dz &\quad \text{for all}\quad& t>0, y\in  \X \label{gdm-eq:f}\\
&\partial_t g(t,x,y)= \Delta_y g(t,x,y) -\lambda(x,y)g(t,x,y) + \mu_1 D(x,y)f(t,x) &\quad \text{for all}\quad& t>0, (x,y)\in \X^2\label{gdm-eq:g}\\
&\nabla_y g(t,x,y)\cdot \Vec{n}=0, &\quad \text{for all}\quad& t>0, x\in \X ,y\in \partial \X \label{gdm-eq:bc-g} \\
& f(0,x)=f_0(x) &\quad \text{for all}\quad& x\in  \X\label{gdm-eq:ic-f}\\
& g(0,x,y)=0 &\quad \text{for all}\quad&  (x,y)\in  \X^2\label{gdm-eq:ic-g}
\end{align}

In the first two subsections we prove that such problem has a non-trivial solution which is unique in a certain functional space.  
Then, we derive  a reduced version of the model that may have some practical interest to exhibit some relevant parameters that  explain exponential explosion of the solution.

\subsection{Existence of a non trivial solution}
From the equation satisfied by  mature plants $f$ \eqref{gdm-eq:f} -\eqref{gdm-eq:ic-f}, we have for  $x\in \X$: 
$$ f(t,x)=f_0(x)+\int_{0}^t\int_{\X}\lambda(z,x)g(s,z,x)\,dz ds,$$
and  we then have a single equation for the moving seed population for $(x,y) \in \X^2$, namely
$$\partial_t g(t,x,y)= \Delta_y g(t,x,y) -\lambda(x,y)g(t,x,y) + \mu_1 D(x,y)\left(f_0(x)+\int_{0}^t\int_{\X}\lambda(z,x)g(s,z,x)\,dz ds \right).$$

Observe that the above equation has some degeneracy in the variable $x$ that makes its resolution non standard. Without the nonlocal term, such type of equation are usually called ultra parabolic equation in the literature \cite{wu2006elliptic}. 

A possible way to construct a solution, is to use a viscosity solution approach. Namely, let $\eps>0$, and introduce the approximation operator 
\[ 
\Lep := \eps\Delta_x + \Delta_y,
\]
and the associated approximation problem:
\begin{align}
&\partial_t f_\eps(t,x)= \int_{\X} \lambda(z,x)g_\eps(t,z,x)\,dz &\quad \text{for all}\quad& t>0, x\in  \X,\label{gdm-eq:feps}\\
&\partial_t g_\eps(t,x,y)= \opLep{g_\eps}(t,x,y) -\lambda(x,y)g_\eps(t,x,y) + \mu_1 D(x,y)f_\eps(t,x) &\quad \text{for all}\quad& t>0, (x,y)\in \X^2,\label{gdm-eq:geps}\\
&\nabla_{x,y} g_\eps(t,x,y)\cdot \Vec{\mathfrak{n}}=0, &\quad \text{for all}\quad& t>0, (x,y)\in \partial \X^2, \label{gdm-eq:bc-geps}\\
& f(0,x)=f_0(x) &\quad \text{for all}\quad& x\in  \X,\label{gdm-eq:ic-feps}\\
& g(0,x,y)=0 &\quad \text{for all}\quad&  (x,y)\in  \X^2\label{gdm-eq:ic-geps},
\end{align}
  where $\Vec{\mathfrak{n}}$ is a normal vector of $\partial (\X^2)$, i.e.  $\Vec{\mathfrak{n}}=(\Vec{n'},\Vec{n})$ with $\Vec{n}$ a normal vector of $\partial \X$.

For any $\eps>0$, provide $\lambda$ and $D$ are smooth functions, we will show that the above system has indeed a solution.  

\begin{proposition}\label{gdm-prop:reg}
Let $\X$ be a smooth bounded domain, with at least a  $C^1$ boundary. Assume further that $f_0\in H^1(\X)$, $f_0\ge 0$, $\lambda, D\in W^{1,\infty}(\X^2)$ non negative functions. Then for all $T>0$, there exists a couple of function $(f_\eps,g_\eps)$ satisfying  $f_\eps\in C^{1}((0,T),H^1(\X))$, $g_\eps \in C^{1}((0,T),H^1(\X)\times H^1(\X))$ such that $f_\eps$ and $g_\eps$ are positive and  $(f_\eps,g_\eps)$ is a solution to 
 \begin{align*}
&\partial_t f_\eps(t,x)= \int_{\X} \lambda(z,x)g_\eps(t,z,x)\,dz &\quad \text{for all}\quad& 0<t<T, x\in  \X,\\
&\partial_t g_\eps(t,x,y)= \opLep{g_\eps}(t,x,y) -\lambda(x,y)g_\eps(t,x,y) + \mu_1 D(x,y)f_\eps(t,x) &\quad \text{for all}\quad& 0<t<T, (x,y)\in \X^2,\\
&\nabla_{x,y} g_\eps(t,x,y)\cdot \Vec{\mathfrak{n}}=0, &\quad \text{for all}\quad& 0<t<T, (x,y)\in \partial \X^2, \\
& f_\eps(0,x)=f_0(x) &\quad \text{for all}\quad& x\in  \X,\\
& g_\eps(0,x,y)=0 &\quad \text{for all}\quad&  (x,y)\in  \X^2.
\end{align*}

\end{proposition}
  
Having the existence of a positive solution of the regularised problem at hand, the trick is now to see whether we can extract a solution of the sequence $f_\eps,g_\eps$ as $\eps \to 0$. To do so, we first establish a uniform estimate independent of $\eps$.

\begin{proposition}\label{gdm-prop:esti1}
Let $\X$ be a smooth bounded domain, with at least a  $C^1$ boundary. Assume further that $f_0\in H^1(\X)$, $f_0\ge 0$, $\lambda,D \in W^{1,\infty}(\X^2)$ are non negatives functions.   
Let $f_\eps,g_\eps$  be positive functions satisfying  $f_\eps\in C^{1}((0,T),H^1(\X))$, $g_\eps \in C^{1}((0,T),H^1(\X)\times H^1(\X))$ and such that \mbox{$(f_\eps,g_\eps)$} is a solution to \eqref{gdm-eq:feps}--\eqref{gdm-eq:ic-geps}. Then $f_\eps$ and $g_\eps$ are increasing in time, and 
there exists $C_0,C_1$ independent of $\eps$ such that 
$$\|f_\eps\|_2\le \|f_0\|_2 e^{C_0t^2} \qquad \text{ and } \qquad \|g_\eps\|_2 \le C_1 \|f_0\|_2 \int_{0}^t e^{C_0\tau^2}\,d\tau.$$
\end{proposition}

\begin{proof}
For convenience, in the rest of the proof we drop the subscript $\eps$ of $f_\eps,g_\eps$ and consider $f,g$.  
Since $f,g$ are positive, due to \eqref{gdm-eq:feps} the monotone increasing character of $f$ is trivial. The monotone character of $g$ is a consequence of the monotone character of $f$ and the parabolic comparison principle. Indeed,  for any real $h>0$, let us define $w_h(t,x,y):=g(t+h,x,y)-g(t,x,y)$. By straightforward computation, since $f$ is monotone increasing in time, we can check that for any $h>0$, $w_h$ satisfies for all $t>0, (x,y) \in \X^2$:

\begin{align*}
&\partial_t  w_h(t,x,y)=  \opLep{w_h}(t,x,y) -\lambda(x,y)w_h(t,x,y) + \mu_1 D(x,y)(f(t+h,x)-f(t,x))\ge  \opLep{w_h}(t,x,y) -\lambda(x,y)w_h(t,x,y)\\
&\nabla_{x,y} w_h(t,x,y) \cdot \Vec{\mathfrak{n}}=0, \\
& w_h(0,x,y)=g(h,x,y)\ge 0.
\end{align*}
 As a consequence, by the Parabolic maximum principle, we get $w_h>0$ for all $t$, meaning that $g(t+h,x,y)\ge g(t,x,y)$. The argument being true for any $h>0$, this means that $g$ is monotonic increasing in time.
 
Let us now obtain the estimate. First, let us multiply \eqref{gdm-eq:geps} by g and integrate over $\X^2$. We then obtain,

$$\frac{1}{2}\frac{d}{dt}\left(\|g\|_2^2\right)= \int_{X^2}g \Delta_y g + \eps \int_{X^2} g \Delta_x g  -\int_{\X^2}\lambda (x,y)g^2 + \mu_1\int_{X^2}D(x,y)f(t,x)g(t,x,y)$$
Since $g$ satisfies \eqref{gdm-eq:bc-geps} and $\partial \X^2 = \partial \X \times \X \cup \X\times \partial \X \cup \partial\X\times\partial \X$, integrating by part the two first integral of the right hand side, we then get 

$$\frac{1}{2}\frac{d}{dt}\left(\|g\|_2^2\right)= -\int_{\X^2} |\nabla_y g|^2 - \eps \int_{\X^2} |\nabla_x g|^2  -\int_{\X^2}\lambda (x,y)g^2 + \mu_1\int_{\X^2}D(x,y)f(t,x)g(t,x,y)$$
and thus, using Cauchy-Schwarz inequality, 

$$\frac{1}{2}\frac{d}{dt}\left(\|g\|_2^2\right)\le \mu_1\int_{\X^2}D(x,y)f(t,x)g(t,x,y)\le \mu_1 \|g\|_2\left(\int_{\X}f^2(t,x)\left(\int_{\X}D^2(x,y)\,dy\right)dx\right)^{\frac{1}{2}}\le \mu_1 \sqrt{\|\bar D\|_{\infty}} \|g\|_2\|f\|_2,$$
By simplifying by $\|g\|_2$, and setting $C_1:=\mu_1\sqrt{\|\bar D\|_{\infty}}$ we get,
 $$\frac{d}{dt}\left(\|g\|_2\right)\le C_1\|f\|_2,$$
and therefore 
\begin{equation}\label{gdm-eq:esti-1}
\|g\|_{2}(t) \le  C_1\int_{0}^{t}\|f\|_2(\tau)\,d\tau.
\end{equation}

On the other hand, by multiplying by $f$ the equation satisfied by $f$ and integrating it over $\X$, we then get, using Cauchy-Schwarz and Jensen's inequality,
$$ \frac{1}{2}\frac{d}{dt}\int_{\X} f^{2}(t,x)\,dx\le \bar \lambda \|f\|_2 \|g\|_2.  $$  
Since $f>0$ for all $t$, dividing the above inequality by $\|f\|_2^2$ and using the estimate \eqref{gdm-eq:esti-1} yield, after simplification,
$$\frac{d}{dt}(\ln(\|f\|_2)\le \bar \lambda \mu_1 \sqrt{\|\bar D\|_{\infty}}\int_{0}^{t}\frac{\|f\|_2(\tau)}{\|f\|_2(t)} \,d\tau.$$ 

Since $f$ is monotone increasing and positive for all $t>\tau>0$, $\frac{\|f\|_2(\tau)}{\|f\|_2(t)} \le 1$ and we get
$$ \frac{d}{dt}(\ln(\|f\|_2)\le \bar \lambda \mu_1 \sqrt{\|\bar D\|_{\infty}} t,$$
which then implies that 
\begin{equation} \label{gdm-eq:esti-f}
\|f\|_2\le \|f_0\|_2 e^{C_0t^2},
\end{equation}
with $C_0:=\frac{1}{2}\bar \lambda C_1$. The estimates on $g$  is a straightforward  consequence of \eqref{gdm-eq:esti-1} that is
\begin{equation} \label{gdm-eq:esti-g}
\|g\|_2\le  C_1\|f_0\|_2 \int_{0}^{t}e^{C_0\tau^2}\,d\tau.
\end{equation}
\end{proof}

For any sequence $\eps_n \to 0$, the estimate is unfortunately insufficient to make sequences  $(f_{\eps_n},g_{\eps_n})_{n\in \N}$ pre-compact in a reasonable functional space.  To gain some compactness, provided some extra assumptions on $\lambda$ and $D$,  we obtain $H^1$ bounds independent of $\eps$. Namely, we have 
 
\begin{proposition}\label{gdm-prop:esti2-reg}
Let $\X$ be a smooth bounded domain, with at least a  $C^1$ boundary. Assume further that $f_0\in H^1(\X)$, $f_0\ge 0$, $\lambda,D \in W^{1,\infty}(\X^2)$ are non negative functions. 
 Let $f_\eps,g_\eps$  be positive functions satisfying  $f_\eps\in C^{1}((0,T),H^1(\X))$, $g_\eps \in C^{1}((0,T),H^1(\X)\times H^1(\X))$ and such that $(f_\eps,g_\eps)$ is a solution to \eqref{gdm-eq:feps}--\eqref{gdm-eq:ic-geps}. 

Then there exists  positive constants $\mathcal{C}_0,\mathcal{C}_1, \mathcal{C}_2,\mathcal{C}_3, \mathcal{C}_4, \mathcal{C}_5$ independent of $\eps$  such that  
$$ \|f_\eps\|_{H^1}\le \|\nabla f_0\|_2+  \|f_0\|_2\left(  e^{C_0t^2}+ \mathcal{C}_0 \int_{0}^{t}\int_{0}^{\tau} e^{C_0s^2}\,ds d\tau +  \mathcal{C}_1 \int_{0}^{t}\int_{0}^{\tau}\int_{0}^{s}e^{C_0\sigma^2}\,d\sigma dsd\tau\right). $$
 and  
 \begin{multline*}
 \|g_\eps\|_{H^1} \le C_1 \|\nabla f_0 \| t + \|f_0\|_2 \left( \mathcal{C}_2  \int_{0}^t e^{C_0\tau^2}\,d\tau + \mathcal{C}_3 \int_{0}^{t}\int_{0}^{\tau} e^{C_0s^2}\,ds d\tau \right.\\+ \left. \mathcal{C}_4 \int_{0}^{t}\int_{0}^{\tau}\int_{0}^{s}e^{C_0\sigma^2}\,d\sigma dsd\tau + \mathcal{C}_5 \int_{0}^{t}\int_{0}^{\tau} \int_{0}^{s}\int_{0}^{\sigma} e^{C_0\omega^2}\,d\omega d\sigma ds d\tau\right).
 \end{multline*}
 Moreover, for all $T>0$, along any sequence $\eps_n \to 0$  there exists a subsequence $(\eps_{n_{k}})_{k\in \N}$ such that 
\[  f_{\eps_{n_k}}\ge f_0  \quad\text{ and }\quad\|g_{\eps_{n_k}}\|_2(t)\ge \frac{1}{2}\|g_1\|_2(t) \quad \text{ for }\quad 0<t<T,
 \]
 where $g_1$ is the unique solution of the ultra-parabolic equation
 \begin{align*}
&\partial_t g_{1}(t,x,y)= \Delta_y g_1(t,x,y) -\lambda(x,y)g_{1}(t,x,y) + \mu_1 D(x,y)f_0(x) &\quad \text{for all}\quad& t>0, (x,y)\in \X^2,\\
&\nabla_{y} g_{1}(t,x,y)\cdot\Vec{n}=0, &\quad \text{for all}\quad& t>0, x\in\X,y\in \partial \X, \\
& g_{1}(0,x,y)=0 &\quad \text{for all}\quad&  (x,y)\in  \X^2.
\end{align*}
\end{proposition}

The proof of this proposition is rather standard knowing the estimate of Proposition \ref{gdm-prop:esti1} and can be founded in the appendix.

From these $H^1$ estimates, for all $T>0$,  we can check that $g_\eps$ is uniformly bounded in $H^{1}((0,T),H^1(\X^2))$ independently of $\eps$. As well, we get that  $f_\eps$ is  uniformly bounded in $H^{1}((0,T),H^1(\X))$ independently of $\eps$.
Therefore, along any sequence $\eps_n \to 0$, by a diagonal extraction process, we can extract a convergent subsequence $(f_n,g_n)_{n,\in \N}$ such that $f_n,g_n \to f^*,g^*$ in $ C((0,T), L^{2}(\X)), C((0,T), L^2(\X^2))$ and  such that  $(f^*,g^*)$ satisfies \eqref{gdm-eq:f} -- \eqref{gdm-eq:ic-g} in the sense of distribution. Up to  extraction of a subsequences, we also have  
\[
f^*\ge f_0,\qquad\text{ and }\qquad \|g^*\|_2(t)\ge \frac{1}{2}\|g_1\|_2(t), \quad \text{ for }\quad t>0.
\]
By using the partial regularity of the singular system \eqref{gdm-eq:f} -- \eqref{gdm-eq:ic-g}, using Sobolev embedding and bootstrap arguments, we  conclude that $f^*\in C^1((0,T), H^1(\X)),g^*\in C^1((0,T), H^1(\X^2))$.

To complete the construction, we need to prove the singular system  \eqref{gdm-eq:f} -- \eqref{gdm-eq:ic-g} have some uniqueness of the solution. 

\subsection*{Uniqueness of solutions}

Here we show that the solution obtain is unique in $H^1$. More precisely, we show 

\begin{proposition}\label{gdm-prop:uniq}
Let $\X$ be a smooth bounded domain, with at least a  $C^1$ boundary. Assume further that $\lambda, D \in W^{1,\infty}(\X^2)$ and $\lambda, D$ are non negative functions.  Then  the system \eqref{gdm-eq:f}--\eqref{gdm-eq:ic-g} has a unique solution $f,g$ such that   $f\in C^{1}((0,T),H^1(\X))$, $g \in C^{1}((0,T),H^1(\X)\times H^1(\X))$. 
\end{proposition}

\begin{proof}
Let us assume that there exists $(f_1,g_1)$ and $(f_2,g_2)$ belonging to $C^1(\R^+, H^1(X))\times C^1(\R^+, H^1(\X^2))$ that are two  solution of \eqref{gdm-eq:f} -- \eqref{gdm-eq:ic-g}. Thanks to the regularity of $\lambda$ and $D$, the functions  $f_i,g_i$ are continuous in $x\in \X, (x,y)\X^2$ for all $t>0$.
 Let us define $w:= g_1 -g_2$, then thanks to \eqref{gdm-eq:g}, we can check that $w$ satisfies
 \begin{align*}
&\partial_t  w(t,x,y)=  \Delta_y w(t,x,y) -\lambda(x,y)w (t,x,y) + \mu_1 D(x,y) \int_{0}^t\int_{\X} \lambda(z,x) w(s,z,x)\,dz\\
&\nabla_y w(t,x,y)\cdot \Vec{n}=0, \\
& w(0,x,y)=0.
\end{align*}
For any  $\gamma \in \R$,  let us consider $\tilde w:=e^{\gamma t}w $. A trivial computation shows that $\tilde w$ satisfies

\begin{align*}
&\partial_t  \tilde w(t,x,y)=  \Delta_y \tilde w(t,x,y) -(\lambda(x,y)-\gamma)\tilde w (t,x,y) + \mu_1 D(x,y) \int_{0}^te^{\gamma(t-s)}\int_{\X} \lambda(z,x) \tilde w(s,z,x)\,dz\\
&\nabla_y \tilde w(t,x,y)\cdot \Vec{n}=0, \\
& \tilde w(0,x,y)=0.
\end{align*}
Take now $\eps>0$ and consider the function $h_\eps(t,x,y):=\tilde w(t,x,y)+\eps$. At $t=0$, $h_\eps(0,x,y)=\eps>0$ and by continuity, $h_\eps$ is positive for later times, says for $t\le t_\eps$. 
We claim that 
\begin{claim}
    $h_\eps(t,x,y)>0$ for all $t>0, (x,y)\in \X^2$.
\end{claim}
Assume for the moment that the claim holds true, we are then done. Indeed, from the above claim we deduce that 
$\tilde w(t,x,y)\ge -\eps$ for all $t>0, (x,y)\in \X^2$ and since $\eps$ is arbitrary chosen  this implies  $\tilde w(t,x,y)\ge 0$ for all $t>0, (x,y)\in \X^2$. By linearity this argument holds true  for  $-\tilde w$ and thus $\tilde w\equiv 0$
for all $t>0,(x,y)\in \X^2$ showing thus the uniqueness of the solution.
\end{proof}
Let us now prove the claim. 
\begin{proof} 
The result will follow from a form of maximum principle. First,  
a quick computation shows that $h_\eps$ satisfies
\begin{multline*}
\partial_t   h_\eps(t,x,y)=  \Delta_y h_\eps(t,x,y) -(\lambda(x,y)-\gamma)h_\eps (t,x,y) + \mu_1 D(x,y) \int_{0}^te^{\gamma(t-s)}\int_{\X} \lambda(z,x) h_\eps(s,z,x)\,dz\\ + \eps( \lambda(x,y)-\gamma) -\eps D(x,y)\bar \lambda(x)\int_{0}^te^{\gamma(t-s)}\,ds. 
\end{multline*}
By taking $\gamma<<-1$  so that 
$\gamma^2 \ge  \|D\|_{\infty}\|\bar \lambda\|_{\infty}$
we get, 
\begin{align*}
&\partial_t   h_\eps(t,x,y)\ge  \Delta_y h_\eps(t,x,y) -(\lambda(x,y)-\gamma)h_\eps (t,x,y) + \mu_1 D(x,y) \int_{0}^te^{\gamma(t-s)}\int_{\X} \lambda(z,x) h_\eps(s,z,x)\,dz\\
&\nabla_y  h_\eps(t,x,y)\cdot\Vec{n}=0, \\
&  h_\eps(0,x,y)=\eps.
\end{align*}
Recall that $h_\eps$ is a smooth function in $t$ and continuous in $x,y$ and that $h_\eps(0,x,y) =\eps>0$.  So by continuity there exists  $t_\eps>0$ such that $h_\eps>0$ for all $t\in(0,t_\eps)$ and $(x,y)\in \X^2$. Let us define 
$$t^*:= \sup\{t>0, |\, h(s,x,y)> 0, \forall s\le t, (x,y)\in \X^2\}.$$
Now, if $t^*=+\infty$ we are done.  So assume by contradiction that $t^*<+\infty$. 
By definition of $t^*$, since $\mu_1,D,\lambda $ and $h_\eps$ are non negative quantities we have for all $t\le t^*$,
\begin{align*}
&\partial_t   h_\eps(t,x,y)\ge  \Delta_y h_\eps(t,x,y) -(\lambda(x,y)-\gamma)h_\eps (t,x,y)\ge \Delta_y h_\eps(t,x,y) -\sigma h_\eps (t,x,y) \\
&\nabla_y  h_\eps(t,x,y)\cdot\Vec{n}=0, \\
&  h_\eps(0,x,y)=\eps.
\end{align*}
where $\sigma:=|\lambda|_{\infty}+|\gamma|$.
To obtain a contradiction, our aim is to show that $h_\eps(t^*,x,y)>0$ for all $(x,y)\in \bar \X^2$.
To do so, let us denote by $\o (t,x,y):=\eps e^{-\sigma t}-h_\eps(t,x)$ then a quick computation shows that 

\begin{align*}
&\partial_t   \o(t,x,y) \le  \Delta_y \o(t,x,y) -\sigma \o(t,x,y) &\text{ for  all }&\quad t\in (0,t^*), x\in \X y \in\X\\
&\nabla_y\o(t,x,y)\cdot \Vec{n} =0  &\text{ for  all }&\quad t\in (0,t^*), x\in \X y \in\partial\X\\
&\o(0,x,y)=0
\end{align*}
Let us now multiply by $\o^+ $ the above inequality and integrate over $\X^2$. After a straightforward computation we get,

$$\frac{1}{2}\frac{d}{dt}\left(\int_{\X^2} (\o^+)^2(t,x,y)\,dxdy\right) \le -\int_{\X^2} |\nabla \o^+|^2(t,x,y)\,dydx-\sigma \int_{\X^2}(\o^+)^2(t,x,y) \le 0.$$
Hence, for all $t\le t^*$, we have 
$$ 0\le \int_{\X^2} (\o^+)^2(t,x,y)\,dxdy \le \int_{\X^2}(\o^+)^2(0,x,y)\,dxdy=0.$$

Therefore  we achieve for all $t\le t^*,$ $\o^+(t,x,y)=0$ for almost every $(x,y) \in \X^2$ meaning in particular that 
$$h_\eps(t^*,x,y)\ge \eps e^{-\sigma t^*}>0 \quad \text{almost everywhere}.$$
From the  continuity of $h_\eps $, it then follows that the inequality holds true for all $(x,y)\in\X^2$ and thus
$$h_\eps(t^*,x,y)\ge \eps e^{-\sigma t^*}>0 \quad \text{for all}\quad (x,y)\in \X^2.$$
From the later by using the continuity of $h_\eps$ with respect to time and space, we can then find  $t'>t^*$ such that
for all $t\le t'$, $h_\eps(t,x,y)>0$ contradicting the definition of $t^*$.
 
\end{proof}

\subsection{A reduced model}
From a population dynamics point of view, a pertinent quantity to look at, is $$\bar g(t,y):=\int_{\X}g(t,x,y)\,dx$$ representing the population of all seeds at the position $y$.  A quick computation shows that 
\begin{align*}
&\partial_t f(t,x)=\int_{\X}\lambda(z,x)g(t,z,x)\,dz &\quad \text{ for all } \quad t>0, x\in \X,\\
&\partial_t \bar g(t,y)= \Delta_y \bar g(t,y) -\int_{ \X}\lambda(x,y)g(t,x,y)\,dx  + \mu_1 \int_{\X}D(x,y)f(t,x)\,dx &\quad \text{ for all } \quad t>0, y\in \X,
\end{align*}
or equivalently 
$$\partial_t \bar g(t,y)= \Delta_y\bar g(t,y) -\int_{\X}\lambda(x,y)g(t,x,y)\,dx + \mu_1 \int_{\X}D(x,y)\left(f_0(x)+\int_{0}^t\int_{\X}\lambda(z,x)g(s,z,x)\,dz  ds \right)\, dx,$$
which for seed's maturity rates $\lambda(x,y)$ only dependent of the deposit position, i.e. $\lambda(x,y)=\lambda(y)$,  reduces
to 
\begin{align*}
&\partial_t f(t,x)=\lambda(x)\bar g(t,x) &\quad \text{ for all } \quad t>0, x\in \X\\
&\partial_t \bar g(t,y)= \Delta_y \bar g(t,y) -\lambda(y)\bar g(t,y)  + \mu_1 \int_{\X}D(x,y)f(t,x)\,dx &\quad \text{ for all } \quad t>0, y\in \X.
\end{align*}
This can also be rewritten as
$$ \partial_t \bar g(t,y)= \Delta_y \bar g(t,y) -\lambda(y)\bar g(t,y)  + \mu_1 \int_{0}^t\int_{\X}D(x,y)\lambda(x)\bar g(s,x)\,dxds + F_0(y) \quad \text{ for all } \quad t>0, y\in \X,$$
where $$F_0(y):=\mu_1 \int_{\X}D(x,y)f_0(x)\,dx.$$
By using that $D$ is a probability density, and since $x$ and $y$ are interchangeable variable we may write the above equation as
$$ \partial_t \bar g(t,y)= \Delta_y \bar g(t,y) + \lambda(y)\left(\mu_1\int_{0}^t\bar g(s,y)\,ds - \bar g(t,y)\right)  + \mu_1 \int_{0}^t\int_{\X}D(x,y)\left[\lambda(x)\bar g(s,x) - \lambda(y)\bar g(s,y)\right]\,dxds + F_0(y). $$

Observe that a first natural time $t_0 $ appears  above which the population of moving seed  will increase  exponentially fast, $t_0$ being the times for which the "cumulative production rate of seed" by plant $\ds{\int_{0}^t \frac{\bar g(s,y)}{\bar g(t,y)}\,ds}$ exceeds $\frac{1}{\mu_1}$.

\appendix
\section{Rigorous definition of the model}
\subsection{Path-wise construction}
We want to rigorously define a process $\lbrace \nuvec_t: t \geq 0 \rbrace$ with generator $\L$ given by \eqref{defgen}. In order to do so, we consider the process $\lbrace \nuvec_t: t \geq 0 \rbrace$ as an element of the path-space of right-continuous process $\D([0,\infty), \M_F)$, taking values in $M_F := \M_F(\bar{\X}) \times \M_F(\bar{\X}^2)$. In the vein of \cite{fournier_microscopic_2004}, we will construct this process in terms of Poisson measures. Hence, we introduce the following probabilistic objects:

\begin{definition}
Let $(\Omega, \f, \P)$ be a sufficiently large probability space. On this probability space we consider the following independent random elements:
\begin{description}
\item[i)] An $\M_F$-valued random variable $\nuvec_0$ of the form $(\nu_p(0), \nu_s(0))$, i.e., the initial distribution of plants and seeds.
\item[ii)] A Poisson random measures $Q_{dep}(ds,di,d\theta)$ on $[0, \infty) \times \N \times \R_+ $, with intensity measure:
\begin{equation*}
l(ds) \otimes (\sum_{j \geq 1}  \delta_j (di)) \otimes l(d\theta),  
\end{equation*}
where $l$ denotes the Lebesgue measure on $\R_+$.
\item[iii)] A Poisson random measure $Q_{dis}(ds,dk,dy,di,d\theta)$ on $[0, \infty) \times \N \times \bar{\X} \times \N \times \R_+ $, with intensity measure:
\begin{equation*}
l(ds) \otimes (\sum_{j \geq 1}  p(j) \delta_j (dk)) \otimes \bar{D}(dy) \otimes (\sum_{j \geq 1}  \delta_j (di)) \otimes d\theta,    
\end{equation*}
where $p$ and $\bar{D}$ are given as in Assumption \ref{AssPherates}.
\end{description}
Moreover, we enlarge the original probability space  $(\Omega, \f, \P)$ to the filtered probability space  $(\Omega, \f, \f_t,  \P)$ , where $( \f_t )_{t \geq 0}$ is the canonical filtration generated by these processes.
\end{definition}
We then introduce the following measures which are useful representing the main events as follows: 
\begin{itemize}
    \item Seeds maturing and becoming plants
\begin{align*}
	\nuvec_{dep}(t) = \int_{[0,t] \times \N^* \times \R_+} \left( \delta_{y_i(s-)}, - \delta_{x_i(s-),y_i(s-)}  \right)\, \indic{i \leq N_s(s^-)} \indic{\theta \leq \lambda(y_i(s-), z_i(s-))} Q_{dep}(ds\, di\, d\theta).
\end{align*}
    \item Plants dispersing new groups of seeds
\begin{align*}
	\nuvec_{dis}(t) = \sum_{k} p(k) \int_{[0,t] \times \X \times \N \times \R_+ } \left( 0, k \delta_{x_i(s-),y}  \right)\, \indic{i \leq N_p(s-)} \indic{\theta \leq D(x_i(s-), y)} Q_{dis}(ds,dy,di,d\theta).
\end{align*}
\end{itemize}

We are now able to define our process as a solution of the following integral equation:

\begin{definition}\label{DefProc}
An $\f_t$-adapted stochastic process $\left(\nuvec_t\right)_{t \geq 0} = (\nu_p(t), \nu_s(t):  t \geq 0)$ is called measure-valued GDM process, if we have that almost surely:
\begin{align}\label{EquYt}
&\langle \phivec, \nuvec(t)  \rangle = \langle \phivec, \nuvec(0) \rangle + \langle \phivec, \nuvec_{dep}(t)  \rangle + \langle \phivec, \nuvec_{dis}(t)  \rangle  + \int_{0}^{t} \left( \langle L \phi_{sp}, \nuvec(s) \rangle + \langle L \phi_{s}, \nu_s(s) \rangle \right) \, ds \nonumber \\ 
&+ \int_{0}^{t} \sum_{i=1}^{N_p(s-)} \sum_{j=1}^{N_s(s-)} \sigma(z_j(s-))  \nabla_z \phi_{s,p}(x_i(s),y_j(s-), z_j(s-))  \, \cdot dW^{i}_s +\int_{0}^{t} \sum_{j=1}^{N_s(s-)} \sigma(z_j(s-)) \nabla_z \phi_{s}(y_i(s-), z_i(s-))  \, dx \, \cdot dW^{i}_s, 
\end{align}
for all $t \geq 0$, and all $ \phivec =\lbrace \phi_p, \phi_s  \rbrace$.
\end{definition}

The following proposition shows that the process given by \eqref{EquYt} coincides with the process with infinitesimal generator $\L$.

\begin{proposition}\label{PropGenerator}
Let $\left(\nuvec_t\right)_{t \geq 0} = (\nu_p(t), \nu_s(t):  t \geq 0)$ be a solution of \eqref{EquYt} such that for all $T>0$
\begin{equation}
    \E \left(  \sup_{t \in [0,T]} (\langle 1, \nu_p(t) \rangle + \langle 1, \nu_s(t) \rangle)^2 \right) < \infty.
\end{equation}
Then, under Assumption \ref{AssPherates}, the process $\left(\nuvec_t\right)_{t \geq 0}$ is Markov with infinitesimal generator $\L$ given by \eqref{defgen}.
\end{proposition}

\begin{proof}
The proof of this proposition is standard. We refer the reader to \cite{fournier_microscopic_2004} or \cite{ayala2022measure} for a proof in similar contexts.
\end{proof}

It remains to show that starting from controlled initial conditions at time zero, we can extend this control for later times.

\begin{proposition}\label{Propcontrolp}
Let $\nuvec_0 = (\nu_p(0), \nu_s(0))$ be such that for some $p \geq 1$ we have
\begin{equation}
    \E \left(  (\langle 1, \nu_p(0) \rangle + \langle 1, \nu_s(0) \rangle)^p \right) < \infty.
\end{equation}
Assume that the counting distribution has finite $p$-th moment:
\begin{equation}\label{pmomentcounting}
  \mu_p :=  \sum_{\kappa} p(\kappa) \kappa^p < \infty.
\end{equation}
Then, under Assumption \ref{AssPherates},  the process $\left(\nuvec_t\right)_{t \geq 0} = (\nu_p(t), \nu_s(t):  t \geq 0)$ satisfies
\begin{equation}\label{pboundt}
    \E \left(  \sup_{t \in [0,T]} (\langle 1, \nu_p(t) \rangle + \langle 1, \nu_s(t) \rangle)^p \right) < \infty.
\end{equation}
\end{proposition}

\begin{proof}
Let us use a stopping time argument to show this proposition. Define the following sequence:
\begin{equation}
    \tau_n = \inf \lbrace t \geq 0: \langle 1, \nu_p(t) \rangle + \langle 1, \nu_s(t) \rangle \geq n  \rbrace.
\end{equation}
Let us also consider $F_p$ given by the choice:
\begin{equation}
    F_p( \nuvec ) =  \langle 1,\nu_p \rangle^p+ \langle 1,\nu_s \rangle^p.
    \end{equation}
Hence, by \eqref{EquYt} we have:
\begin{align}
\sup_{s \in [0,t \wedge \tau_n] } &\langle 1, \nu_p(s) \rangle^p+ \langle 1, \nu_s(s) \rangle^p  \leq  \langle 1, \nu_p(0) \rangle^p+ \langle 1, \nu_s(0) \rangle^p \nonumber \\
&\hspace{-0.8cm}+  \sum_{k} p(k) \int_0^{t \wedge \tau_n} \int_{ \X \times \N \times \R_+}  \left[ (\langle 1, \nu_s(s-) \rangle +k )^p  - \langle 1, \nu_s(s-) \rangle^p \right]  \, \indic{i \leq N_p(s-)} \indic{\theta \leq D(x_i(s-), y)} Q_{dis}(ds,dy,di,d\theta),
\end{align}
where the diffusion terms vanished due to the presences of derivatives of the constant function $1$, and we have dropped the integral term related to maturation since it has a negative contribution.\\

Taking expectation we obtain
\begin{align}
\E \left[ \sup_{s \in [0,t \wedge \tau_n] } \langle 1, \nu_p(s) \rangle^p+ \langle 1, \nu_s(s) \rangle^p  \right]  &\leq  \E \left[ \langle 1, \nu_p(0) \rangle^p+ \langle 1, \nu_s(0) \rangle^p  \right] \nonumber \\
&+  C \E \left[\int_0^{t \wedge \tau_n}  N_p(s-) \sum_{k} p(k)\left[ (\langle 1, \nu_s(s-) \rangle + k )^p  - \langle 1,\nu_s(s-) \rangle^p \right]  \, \right] \, ds,
\end{align}
where the constant $C$ incorporates the bounds given by Assumption \ref{AssPherates}.\\

Using \eqref{pmomentcounting} and the simple inequality $(x+1)^p -x^p \leq C_p (1+x^{p-1})$ we obtain
\begin{align}
\E \left[ \sup_{s \in [0,t \wedge \tau_n] } \langle 1, \nu_p(s) \rangle^p+ \langle 1, \nu_s(s) \rangle^p  \right]  &\leq  \E \left[ \langle 1, \nu_p(0) \rangle^p+ \langle 1, \nu_s(0) \rangle^p  \right] +  C \E \left[\int_0^{t}  N_p(s \wedge \tau_n) \left[ 1+ \langle 1,\nu_s(s \wedge \tau_n ) \rangle^{p-1} \right]  \, \right] \, ds, 
\end{align}
where the constant $C$ has changed its value again. Moreover, using the equality
\begin{equation}
    N_p(s) = \langle 1, \nu_p(s) \rangle,
\end{equation}
we obtain
\begin{align}
\E \left[ \sup_{s \in [0,t \wedge \tau_n] } \langle 1, \nu_p(s) \rangle^p+ \langle 1, \nu_s(s) \rangle^p  \right]  &\leq C_p \left( 1 + \E \left[\int_0^{t}   \langle 1, \nu_p(s) \rangle^p+ \langle 1, \nu_s(s) \rangle^p    \, ds \right]  \right),
\end{align}
where again the constant $C_p$ changed its value incorporating new constants.\\

By Gronwalls inequality we then have
\begin{equation}\label{GronwApp}
\E \left[ \sup_{s \in [0,t \wedge \tau_n] } \langle 1, \nu_p(s) \rangle^p+ \langle 1, \nu_s(s) \rangle^p  \right]\leq   C_p e^{Bt},
\end{equation}
for some constant $B$ independent of $n$. From \eqref{GronwApp} we can deduce that $\tau_n$ goes to a.s. to infinity as $n\to \infty$. We can then apply Fatou's lemma to conclude.\\
\end{proof}

For existence the well-definedness of the process $\nu_t$ one has to construct the process step by step, where the time steps are given by a sequence of jump instants $T_n$ exponentially distributed  with law
\begin{equation}
R(\nu_{n-1}) e^{-R(\nu_{n-1})t},
\end{equation}
and where the total jump rate $R(\nu)$ is bounded by Assumption \ref{AssPherates}. It is then enough to check that the sequence $T_n$ goes to infinity almost surely. This follows from
\begin{equation*}
 \E \left( \sup_{t \leq T} \, \sum_{\alpha \in \X} \, \langle \nu_\alpha(t),1 \rangle_{\D_\alpha}^p \right) < \infty,  
\end{equation*}
which is a consequence of \eqref{pboundt} when $p=1$. However, for the martingale characterization of $\lbrace \nuvec(t) \rbrace_{t \geq 0}$ higher values of $p$ are needed.

\section{Estimate on the PDE system}
In this section, we prove Proposition \ref{gdm-prop:reg} and  establish some estimates on the solution to the regularised problem.
Let us start with the proof of Proposition \ref{gdm-prop:reg}.
\begin{proof}
Let us introduce the following sequence of functions $(f_n,g_n)_{n\in \N}$, where $(f_0(t,x),g_0(t,x))=(f_0(x),0)$,
and for all $n>0$, $f_n$ and $g_n$ are respectively solution of the following  equations: 
\begin{align}
&\partial_t f_n(t,x)= \int_{\X} \lambda(z,x)g_n(t,z,x)\,dz &\quad \text{for all}\quad& t>0, x\in  \X \label{gdm-eq:f_n-def}\\
&\partial_t g_n(t,x,y)=  \opLep{g_n}(t,x,y) -\lambda(x,y)g_n(t,x,y) + \mu_1 D(x,y)f_{n-1}(t,x) &\quad \text{for all}\quad& t>0, (x,y)\in \X^2\label{gdm-eq:g_n-def}\\
&\nabla_{x,u}g_n(t,x,y)\cdot\Vec{\mathfrak{n}}=0, &\quad \text{for all}\quad& t>0, (x,y)\in \partial \X^2 \\
& f_n(0,x)=f_0(x) &\quad \text{for all}\quad& x\in  \X\\
& g_n(0,x,y)=0 &\quad \text{for all}\quad&  (x,y)\in  \X^2.
\end{align}
Since, for all $n>0$, $g_n$ satisfies a parabolic equation, the sequence $(f_n,g_n)$ is then well defined for all $t$ as soon as for all $n$, we show that $f_n(t,x) \in L^{2}(\X)$ for all $t>0$.
We claim that 

\begin{claim}
For all $n\ge 0$ and $t>0, f_n(t,x)\in L^2(\X)$. 
\end{claim} 
\begin{proof}
To prove such claim we will use a induction argument.  Since by definition $f_0(t,x)=f_0(x)$ for all $t$ the results is trivial for $n=0$. 
Let us assume that for some $n\ge 0$, for all $t>0,$   $f_n(t,\cdot)\in L^{2}(\X)$ and $f_n\ge 0$ let us prove that for all $t>0,f_{n+1}(t,\cdot) \in L^2(\X)$ and $f_{n+1}(t,x)\ge 0$.
Observe that since $f_n\in L^2(\X)$ for all $t>0$,  and $\bar D(x)\in L^{\infty}(\X)$ we have 
$$\int_{\X^2}F_n^2(t,x,y)\,dxdy=\int_{\X^2}\mu_1^2 D^2(x,y)f^2_{n}(x)\, dydx =\mu_1^2 \int_{\X}\left(f^2_n(t,x)\int_{\X}D^2(x,y)\,dy\right)\,dx\le C\int_{\X}f_{n}^2((x)\,dx ,$$
and therefore $g_{n+1}$ satisfies a parabolic equation with a right hand side $F_n(t,x,y) \in L^2(\X^2)$. 
From standard parabolic theory \cite{friedman2008partial,pazy2012semigroups,wu2006elliptic}, since $\X$ is smooth, $g_{n+1}$ is well defined for all $t>0$ and $g_{n+1}\in C^{1}(\R^+,H^{1}(\X^2))$. Now observe that since $g_n\in C^{1}(\R^+,H^{1}(\X^2))$ and $\lambda\in L^{\infty}(\X^2)$, the function 
$\ds{h_{n+1}(t,x):=\int_{\X}\lambda(x,z)g_{n+1}(t,z,x)\,dz}$ belongs to $ C^{1}(\R^+,L^1(\X))$ 
and from the equation satisfied by  $f_{n+1}(t,x)$, we straightforwardly  get that $f_{n+1}$ is well defined for all times $t>0$ and $f_{n+1}\in C^1(\R^+,L^1(\X))$. Moreover, we have  
\begin{equation}\label{gdm-eq:f_nint-def}
f_{n+1}(t,x)=f_0(x) + \int_{0}^t \int_{\X}\lambda(z,x)g_{n+1}(s,z,x)\,dz ds.
\end{equation}
By Young's and Jensen inequalities, we get that 
$$ f_{n+1}^2(t,x) \le C \left(f_0^2(x) + \bar \lambda\int_{0}^t\int_{\X}g_{n+1}^2(s,z,x)\,dsdz\right),  $$ 
and therefore $f_{n+1}(t,\cdot)\in L^2(\X)$ for all $t>0$.
Note that as a side result of the claim we get that for all $n\ge0$ $(f_n(t,x),g_n(t,x,y))$ is well defined for all $t>0$, $x\in \X$,$y\in \X$ and moreover, for all $n,\, g_n \in C^1(\R^+,H^1(\X^2))$ and $f_n-f_0 \in C^2(\R^+,H^1(\X))$.
\end{proof}

Now, to conclude to the existence of a solution, we need to check that $(f_n(t,x),g_n(t,x))_{n\in\N}$ converges in a certain sense. Before that, let us show that the sequences $f_n$ and $g_n$ are monotone increasing
\begin{claim}
For all $n> 0$ and $t>0$, $f_{n+1}(t,x)\ge f_{n}$ and $g_{n+1}(t,x,y)\ge g_n(t,x,y)$ for all $x,y\in\X$.  
\end{claim}

\begin{proof}
Observe by the relation \eqref{gdm-eq:f_nint-def} we get the monotone behaviour of the sequence  $(f_n)_{n\in \N}$ directly from the one of $(g_n)_{n\in \N^*}$.

To obtain the behaviour of the sequences  $f_n, g_n$, we will argue by induction. But before let us observe that thanks to \eqref{gdm-eq:f_nint-def}  we have $f_1\ge f_0$ if we can prove that $g_1\ge 0$. The positivity of $g_1$, is a straightforward consequence of the parabolic maximum principle. Indeed, since, $f_0,D,\lambda$ and $\mu_1$ are positive quantities, by definition of $g_1$, it satisfies for all $t>0, (x,y)\in \X^2$:
  \begin{align*}
&\partial_t g_{1}(t,x,y)=  \opLep {g_{1}}(t,x,y) -\lambda(x,y)g_{1}(t,x,y)+\mu_1D(x,y)f_0(x) \ge  \opLep{g_{1}}(t,x,y) -\lambda(x,y)g_{1}(t,x,y),\\
&\nabla_{x,y}g_1(t,x,y)\cdot\Vec{\mathfrak{n}}=0,  \\
& g_1(0,x,y)=0,
\end{align*} 
which from  the parabolic maximum principle entails that $g_1(t,x,y)\ge 0$ for all $t\ge 0$ and $(x,y)\in \X^2$.

As a consequence $f_1\ge f_0$, we can check that $ w=g_2-g_1$ satisfies for all $t>0, (x,y)\in \X^2$:

  \begin{align*}
&\partial_t w(t,x,y)=  \opLep{w}(t,x,y) -\lambda(x,y)w+\mu_1D(x,y)(f_1(t,x)-f_0(x)) \ge  \opLep{w}(t,x,y) -\lambda(x,y)w(t,x,y),\\
&\nabla_{x,y}w(t,x,y)\cdot\Vec{\mathfrak{n}}=0,  \\
& w(0,x,y)=0,
\end{align*} 
which again by the parabolic maximum principle entails that $w(t,x,y)\ge 0$ for all $t\ge 0$ and $(x,y)\in \X^2$, meaning that $g_2\ge g_1$ and by \eqref{gdm-eq:f_nint-def}, $f_2\ge f_1$.

Let us now argue by induction, and let us assume there exists $k\ge 2$ such that for all $0<n\le k$, $f_n\le f_{n+1}$ and $g_n\le g_{n+1}$.
Let us prove that these inequality are still true for $k+1$. Note that thanks to our induction assumption we only need to show that 
$f_{k+1}\le f_{k+2}$ and $g_{k+1}\le g_{k+2}$.
Define now $w=g_{k+2} -g_{k+1}$. Thanks to the definition of $g_n$ and using that $f_{k+1}\ge f_k$ we deduce that for all $t>0, (x,y)\in \X^2$  
 \begin{align*}
&\partial_t w(t,x,y)=  \opLep{w}(t,x,y) -\lambda(x,y)w+\mu_1D(x,y)(f_{k+1}(t,x)-f_{k}(x)) \ge  \opLep{w}(t,x,y) -\lambda(x,y)w(t,x,y),\\
&\nabla_{x,y}w(t,x,y)\cdot\Vec{\mathfrak{n}}=0, \\
& w(0,x,y)=0.
\end{align*} 
By using once again the parabolic maximum principle, we get $w\ge 0$, for all $t\ge 0$ and $(x,y)\in \X^2$, meaning that $g_{k+2}\ge g_{k+1}$. We then ensure that $f_{k+2}\ge f_{k+1}$ by using \eqref{gdm-eq:f_nint-def}.

Since from the above computation, the induction assumption is true for $k=2$, we are done. 

\end{proof}

We are now in position to show that $(f_n(t,x),g_n(t,x))_{n\in\N}$ converges. From the monotone behaviour of the sequence,to ensure the convergence,  we only need to obtain a uniform upper bound independent of $n$. 

To do so, since $f_n \in L^2(\X)$, let us multiply \eqref{gdm-eq:f_n-def} by $f_n$ and integrate over $\X$. We then get using Jensen inequality

$$
\frac{1}{2} \frac{d}{dt} \int_{\X}f^2_n(t,x)\,dx = \int_{\X} \left (f_n(t,x)\int_{\X}\lambda(z,x)g_n(t,z,x)\,dz\right)\,dx\le \bar \lambda \|f_n\|_{2}\|g_n\|_{2}.
$$ 
Therefore we get
$$
\frac{d}{dt} \|f_n\|_2 \le \bar \lambda \|g_n\|_{2},
$$ 
leading to 
$$ \|f_n\|_2 \le \|f_0\|_0+ \bar \lambda \int_{0}^t\|g_n\|_{2}(s)\,ds.$$
Now let us multiply \eqref{gdm-eq:g_n-def} by $g_n$ and integrate over $\X^2$, a short computation shows that 
$$\frac{d}{dt} \|g_n\|_2 \le C\left(\int_{\X}f^2_{n-1}(t,x)\left(\int_{\X}D^2(x,y)\, dy\right)dx\right)^{\frac{1}{2}}  \le C\|f_{n-1}\|_2. $$

Thus, we get 
$$\|g_n\|_{2}(t)\le C\int_{0}^t\|f_{n-1}\|_2(s)\,ds $$
leading to

$$
 \|f_n\|_2 \le \|f_0\|_2 + \bar \lambda C\int_{0}^t \int_{0}^s\|f_{n-1}\|_{2}(\tau)\,d\tau.
$$

By a recursive argument,  for all $n\ge 0$ we achieve
$$
\|f_n\|_2 \le \|f_0\|_2\sum_{i=0}^{n}\frac{(Ct^2)^i}{2i!},
$$
and so we obtain the bound  
$$ \|f_n\|_2\le \|f_0\|_2e^{Ct^2}.$$
The latter  enforces  also the following  bound on $g_n$
$$
\|g_n\|_{2}  \le C\int_{0}^t\|f_{n-1}\|_2(s)\,ds\le C \|f_0\|_2\int_{0}^te^{Cs^2}\,ds.
$$

To obtain $H^1$ estimate, let us observe that since $\lambda,D\in W^{1,\infty}$, then we can check that $\nabla_x f_n$ and $\nabla_{i}g_n$ with $i=x$ or $i=y$ respectively satisfies

\begin{align*}
&\partial_t \nabla_y f_n(t,y)= \int_{\X} \nabla_y\lambda(z,y)g_n(t,z,y)\,dz+ \int_{\X} \lambda(z,y)(\nabla_yg_n)(t,z,y)\,dz &\quad \text{for all}\quad& t>0, y\in  \X %\label{gdm-eq:gradf_n-def}
\\
&\partial_t \nabla_{i}g_n(t,x,y)=  \opLep{\nabla_{i}g_n}(t,x,y) -\nabla_{i}(\lambda(x,y)g_n(t,x,y)) + \nabla_{i}\left(\mu_1 D(x,y)f_{n-1}(t,x)\right) &\quad \text{for all}\quad& t>0, (x,y)\in \X^2%\label{gdm-eq:gradg_n-def}
\\
%&\nabla_{x,u}\nabla_i(g_n(t,x,y))\cdot\Vec{\mathfrak{n}}=0, &\quad \text{for all}\quad& t>0, (x,y)\in \partial \X^2 \\
& \nabla_yf_n(0,y)=\nabla_yf_0(y) &\quad \text{for all}\quad& y\in  \X\\
& \nabla_{i}g_n(0,x,y)=0 &\quad \text{for all}\quad&  (x,y)\in  \X^2.
\end{align*}
By multiplying the first equality by $\nabla_yf_n$ and the second by $\nabla_ig_n$ and integrate respectively on $\X$ and $\X^2$, we get after straightforward computations using the estimates on $f_n,g_n$ above
\begin{align*}
&\partial_t \|\nabla f_n\|_{2}\le C \|f_0\|_2\int_{0}^te^{Cs^2}\,ds + C\|\nabla_y g_n\|_{2}   &\quad \text{for all}\quad& t>0,  \\
&\partial_t \|\nabla_{y}g_n\|_{2}\le   C \|f_0\|_2\left( e^{Ct^2}+ \int_{0}^te^{Cs^2}\,ds\right)   &\quad \text{for all}\quad& t>0,\\
&\partial_t \|\nabla_{x}g_n\|_{2}\le   C \|f_0\|_2\left( e^{Ct^2}+ \int_{0}^te^{Cs^2}\,ds\right) + C\|\nabla f_{n-1}\|_{2}   &\quad \text{for all}\quad& t>0.
\end{align*}
We then obtain the desired $H^1$ estimates by exploiting the above differential relations and the estimates on $f_n$ and $g_n$.

Now, for all $t\le T$ the sequence $(g_n, f_n-f_0)$ is uniformly bounded in $C^{1,1}((0,T),H^1(\X^2)) \times C^{1,1}((0,T),H^1(\X))$, so, up to extraction of a subsequence, there exists a convergent subsequence in $C^1((0,T),L^2(\X))$ to some functions $(\bar g, \bar f)$ which is a weak solution to 

\begin{align*}
&\partial_t f(t,x)= \int_{\X} \lambda(z,x)g(t,z,x)\,dz &\quad \text{for all}\quad& 0<t<T, x\in  \X\\
&\partial_t  g(t,x,y)=  \opLep{g}(t,x,y) -\lambda(x,y)g(t,x,y) + \mu_1 D(x,y)f(t,x) &\quad \text{for all}\quad& 0<t<T, (x,y)\in \X^2\\
&\nabla_{x,y} g(t,x,y)\cdot\Vec{\mathfrak{n}}=0, &\quad \text{for all}\quad& 0<t<T, (x,y)\in \partial \X^2 \\
& f(0,x)=f_0(x) &\quad \text{for all}\quad& x\in  \X\\
& g(0,x,y)=0 &\quad \text{for all}\quad&  (x,y)\in  \X^2.
\end{align*}
Thanks to elliptic and parabolic regularity estimates, we can check that $\bar g,\bar f$ is indeed a strong solution to the above equation. 
To complete the construction, we need to verify that $\bar g, \bar f$ is non trivial which is straightforward, since $(f_n)_{n\in\N}$ and $(g_n)_{n\in \N}$ are  monotone increasing sequences. 
\end{proof}

Next, we obtain some useful estimate and prove Proposition \ref{gdm-prop:esti2-reg}  that we recall below.

\begin{proposition}
Let $\X$ be a smooth bounded domain, with at least a  $C^1$ boundary. Assume further that $f_0\in H^1(\X)$, $f_0\ge 0$, $\lambda,D \in W^{1,\infty}(\X^2)$ are non negative functions.
 Let $f_\eps,g_\eps$  be positive functions satisfying  $f_\eps\in C^{1}((0,T),H^1(\X))$, $g_\eps \in C^{1}((0,T),H^1(\X)\times H^1(\X))$ and such that $(f_\eps,g_\eps)$ is a solution to \eqref{gdm-eq:feps}--\eqref{gdm-eq:ic-geps}. 

Then there exists  positive constants $\mathcal{C}_0,\mathcal{C}_1, \mathcal{C}_2,\mathcal{C}_3, \mathcal{C}_4, \mathcal{C}_5$ independent of $\eps$  such that  
$$ \|f_\eps\|_{H^1}\le \|\nabla f_0\|_2+  \|f_0\|_2\left(  e^{C_0t^2}+ \mathcal{C}_0 \int_{0}^{t}\int_{0}^{\tau} e^{C_0s^2}\,ds d\tau +  \mathcal{C}_1 \int_{0}^{t}\int_{0}^{\tau}\int_{0}^{s}e^{C_0\sigma^2}\,d\sigma dsd\tau\right), $$
 and  
 \begin{multline*}
 \|g_\eps\|_{H^1} \le C_1 \|\nabla f_0 \| t + \|f_0\|_2 \left( \mathcal{C}_2  \int_{0}^t e^{C_0\tau^2}\,d\tau + \mathcal{C}_3 \int_{0}^{t}\int_{0}^{\tau} e^{C_0s^2}\,ds d\tau \right.\\+ \left. \mathcal{C}_4 \int_{0}^{t}\int_{0}^{\tau}\int_{0}^{s}e^{C_0\sigma^2}\,d\sigma dsd\tau + \mathcal{C}_5 \int_{0}^{t}\int_{0}^{\tau} \int_{0}^{s}\int_{0}^{\sigma} e^{C_0\omega^2}\,d\omega d\sigma ds d\tau\right).
 \end{multline*}
  Moreover, for all $T>0$, along any sequence $\eps_n \to 0$  there exists a subsequence $(\eps_{n_{k}})_{k\in \N}$ such that 
 $$ f_{\eps_{n_k}}\ge f_0  \quad\text{ and }\quad\|g_{\eps_{n_k}}\|_2(t)\ge \frac{1}{2}\|g_1\|_2(t) \quad \text{ for }\quad 0<t<T,$$ where $g_1$ is the unique solution of the ultra-parabolic equation:
 \begin{align*}
&\partial_t g_{1}(t,x,y)= \Delta_y g_1(t,x,y) -\lambda(x,y)g_{1}(t,x,y) + \mu_1 D(x,y)f_0(x) &\quad \text{for all}\quad& t>0, (x,y)\in \X^2,\\
&\nabla_y g_{1}(t,x,y)\cdot\Vec{n}=0, &\quad \text{for all}\quad& t>0, x\in\X,y\in \partial \X, \\
& g_{1}(0,x,y)=0 &\quad \text{for all}\quad&  (x,y)\in  \X^2.
\end{align*}
\end{proposition}

\noindent
\begin{proof}

To obtain now the estimate on the $H^1$ norm, we need to get uniform $L^2$ estimates on the $\nabla f, \nabla_xg$ and $\nabla_y g$.
To do so, let us observe that since $\lambda,D\in W^{1,\infty}$, $f\in H^1(\X)$ and $g \in H^1(\X^2)$, from \eqref{gdm-eq:geps} we can check that $\nabla_y g$ satisfies, in the sense of distribution, 
$$\partial_t \nabla_y g(t,x,y)= \opLep{\nabla_y g}(t,x,y) - \nabla_y \lambda(x,y). g(t,x,y) - \lambda(x,y)\nabla_y g(t,x,y) + \mu_1 f(t,x) \nabla_y D(x,y).$$
By multiplying  the above equation by $\nabla_y g$ and integrating  over $\X^2$, we get, using the coercivity of $\Lep$ and Cauchy-Schwarz inequality,

\begin{align*}
\frac{1}{2} \frac{d}{dt} \int_{\X^2} |\nabla_yg(t,x,y)|^2\,dxdy &= \langle \opLep{\nabla_y g}(t,x,y), \nabla_{y} g(t,x,y)\rangle -  \int_{\X^2} g\nabla_y\lambda  \nabla_y g -  \int_{\X^2} \lambda  |\nabla_y g|^2 + \mu_1 \int_{\X^2} f\nabla_yD  \nabla_y g\\
  & \le  \|\nabla_y g\|_2\left(\|\nabla_y \lambda\|_{\infty} \|g\|_2 + \mu_1  \sqrt{\int_{\X}f^2(t,x)\left(\int_{\X} |\nabla_y D(x,y)|^2 \, dy\right) \,dx}\right),
\end{align*}
 which after simplification by $\|\nabla_y g\|_2$, gives 

$$ 
\frac{d}{dt}\|\nabla_y g\|_2 \le  C_2\|g\|_2+ C_3\|f\|_2,
$$
where $C_2:=\|\lambda\|_{1,\infty}$ and $C_3:= \mu_1 \sqrt{\sup_{x\in \X} \int_{\X} |\nabla_y D(x,y)|^2 \, dy}$.\\
Using the estimates \eqref{gdm-eq:esti-f} and \eqref{gdm-eq:esti-g}, we achieve  

$$ 
\frac{d}{dt}\|\nabla_y g\|_2 \le  C_2 C_1\|f_0\|_2 \int_{0}^{t}e^{C_0\tau^2}\,d\tau + C_3\|f_0\|_2 e^{C_0t^2},
$$
which after integration in time yields the $L^2$ bound

\begin{equation}\label{gdm-eq:esti-grady-g}
\|\nabla_y g\|_2 \le  C_2 C_1\|f_0\|_2 \int_{0}^{t}\int_{0}^{\tau}e^{C_0s^2}\,dsd\tau + C_3\|f_0\|_2 \int_{0}^{t}e^{C_0\tau^2}\,d\tau.
\end{equation}

Let us now obtain a $L^2$ bound for $\nabla f$.  Since $\lambda\in W^{1,\infty}$, $f\in H^1(\X)$ and $g \in H^1(\X^2)$, from \eqref{gdm-eq:feps} we can check that $\nabla f$ satisfies, in the sense of distribution,  
$$\frac{d}{dt}\nabla f (t,y)=\int_{\X}\nabla_y\lambda(z,y) g(t,z,y)\,dy + \int_{\X}\lambda(z,y)\nabla_y g(t,z,y)\,dz.$$
By multiplying the later equality by $\nabla f$ and integrating over $\X$, it yields, using Cauchy-Schwarz inequality and $\lambda \in W^{1,\infty}$,
\begin{align*}
\frac{1}{2}\frac{d}{dt}\int_{\X}|\nabla f (t,y)|^2\,dy&=\int_{\X}\nabla f(t,y)\left(\int_{\X}\nabla_y\lambda(z,y) g(t,z,y)\,dy + \int_{\X}\lambda(z,y)\nabla_y g(t,z,y)\,dz\right)\,dy,\\
& \le \|\nabla f\|_2 \left(  \|\nabla_y\lambda\|_{\infty}\|g\|_2+ \|\lambda\|_{\infty}\|\nabla_y g\|_2\right)\\
& \le \|\nabla f\|_2 C_2\left( \|g\|_2+ \|\nabla_y g\|_2\right).
\end{align*}
After simplification by $\|\nabla f\|_2$, we get 
$$ 
\frac{d}{dt}\|\nabla f\|_2 \le C_2(\|g\|_2+\|\nabla_y g\|_2).
$$

Integrating in time and using \eqref{gdm-eq:esti-g} and \eqref{gdm-eq:esti-grady-g}, then yields  
\begin{equation}\label{gdm-eq:esti-grad-f}
\|\nabla f\|_2 \le  \|\nabla f_0\|_2 + C_2(C_3+C_1)\|f_0\|_2 \int_{0}^{t}\int_{0}^{\tau} e^{C_0s^2}\,ds d\tau +  C_2 C_1\|f_0\|_2 \int_{0}^{t}\int_{0}^{\tau}\int_{0}^{s}e^{C_0\sigma^2}\,d\sigma dsd\tau. 
\end{equation}

Last we obtain a $L^2$ bound for $\nabla_x g$. As for $\nabla_y g$ from \eqref{gdm-eq:geps} we can check that $\nabla_x g$ satisfies, in the sense of distribution, 
$$\partial_t \nabla_x g(t,x,y)= \opLep{\nabla_x g}(t,x,y) - \nabla_x \lambda(x,y). g(t,x,y) - \lambda(x,y)\nabla_x g(t,x,y) + \mu_1 f(t,x) \nabla_x D(x,y) + \mu_1 D(x,y)\nabla f(t,x).$$
By multiplying  the above equation by $\nabla_x g$ and integrating  over $\X^2$, we get, using the coercivity of $\Lep$ and Cauchy-Schwarz inequality,

\begin{align*}
\frac{1}{2} \frac{d}{dt} \int_{\X^2} |\nabla_xg|^2\,dxdy &= \langle \opLep{\nabla_x g}, \nabla_{x} g\rangle -  \int_{\X^2} g\nabla_x\lambda  \nabla_x g -  \int_{\X^2} \lambda  |\nabla_x g|^2 + \mu_1 \int_{\X^2} f\nabla_xD  \nabla_x g + \mu_1 \int_{\X^2}D \nabla f  \nabla_x g\\
  & \le \|\nabla_x g\|_2 \left[ \|\nabla_x \lambda\|_{\infty} \|g\|_2 + \mu_1 \sqrt{\int_{\X}f^2(t,x)\left(\int_{\X} |\nabla_x D(x,y)|^2 \, dy\right) \,dx} +  \mu_1 \sqrt{\|\bar D\|_{\infty}}\|\nabla f\|_2\right]\\
& \le \|\nabla_x g\|_2 \left[ C_2 \|g\|_2 + \mu_1 \sqrt{\int_{\X}f^2(t,x)\left(\int_{\X} |\nabla_x D(x,y)|^2 \, dy\right) \,dx} +  C_1\|\nabla f\|_2\right]  .
\end{align*}
After simplification by $\|\nabla_x g\|_2$, and setting $C_4:= \mu_1\sqrt{\sup_{x\in\X}\left(\int_{\X} |\nabla_x D(x,y)|^2 \, dy\right)}$ we get 
\begin{align*}
\frac{d}{dt}\|\nabla_x g\|_2 &\le \left[ C_2 \|g\|_2 + \mu_1 \sqrt{\int_{\X}f^2(t,x)\left(\int_{\X} |\nabla_x D(x,y)|^2 \, dy\right) \,dx} +  C_1\|\nabla f\|_2\right],\\
&\le \left[ C_2 \|g\|_2 + C_4 \|f\|_2+  C_1\|\nabla f\|_2\right],
\end{align*}

which  after integrating in time and using \eqref{gdm-eq:esti-g}, \eqref{gdm-eq:esti-grad-f}, then yields  
\begin{multline}\label{gdm-eq:esti-gradx-g}
\|\nabla_xg\|_2 \le  C_1\|\nabla f_0\|_2 t +  C_4 \|f_0\|_2 \int_{0}^{t} e^{C_0\tau^2}\, d\tau + C_2C_1 \|f_0\|_2 \int_{0}^{t}\int_{0}^{\tau} e^{C_0s^2}\,ds d\tau \\  
 +   C_1C_2(C_3+C_1)\|f_0\|_2 \int_{0}^{t}\int_{0}^{\tau}\int_{0}^s e^{C_0\sigma^2}\,d\sigma ds d\tau +
 C_2 C_1^2\|f_0\|_2 \int_{0}^{t}\int_{0}^{\tau}\int_{0}^{s}\int_{0}^\sigma e^{C_0\omega^2}\,d\omega d\sigma dsd\tau
.
\end{multline}
Collecting \eqref{gdm-eq:esti-f}, \eqref{gdm-eq:esti-grad-f}, \eqref{gdm-eq:esti-g}, \eqref{gdm-eq:esti-grady-g} and \eqref{gdm-eq:esti-gradx-g}, then get the desired estimate.

To obtain the estimate from below, we just have to observe that since for all $\eps,$ $g_\eps$ is positive, the first assertion is then trivial.  Now since $f_\eps(t,x)\ge f_0$ for all $\eps$, we have for all $t>0, (x,y) \in \X^2$
\[ 
\partial_t g_{\eps}(t,x,y)= \opLep{g_{\eps}}(t,x,y) -\lambda(x,y)g_{\eps}(t,x,y) + \mu_1 D(x,y)f_\eps(t,x) \ge \opLep{g_{\eps}}(t,x,y) -\lambda(x,y)g_{\eps}(t,x,y) + \mu_1 D(x,y)f_0(x),
\]
and for all $\eps$, we have by the parabolic comparison principle, $g_\eps(t,x,y)\ge g_{1,\eps}(t,x,y)$. 
To obtain the last part of the estimate from below, we need to characterised the convergence of $g_{1,\eps}$ with respect to $\eps$. Such characterisation follows as consequence of uniform $H^1$ bounds on $g_{1,\eps}$ that  can be obtained in a similar fashion as the one obtained for $g_\eps$.  The uniqueness of the limit solution follows also from a standard argument on the $L^2$ norm. 
\end{proof}

\bibliographystyle{plain}
		\bibliography{gdm-biblio}	

\begin{thebibliography}{10}

\bibitem{ayala2022measure}
Mario Ayala, Jerome Coville, and Raphael Forien.
\newblock A measure-valued stochastic model for vector-borne viruses.
\newblock {\em arXiv preprint arXiv:2211.04563}, 2022.

\bibitem{branch1993recruitment}
Lyn~C Branch, Diego Villarreal, and Gene~S Fowler.
\newblock Recruitment, dispersal, and group fusion in a declining population of
  the plains vizcacha (lagostomus maximus; chinchillidae).
\newblock {\em Journal of mammalogy}, 74(1):9--20, 1993.

\bibitem{champagnat_individual_2008}
Nicolas Champagnat, R{\'e}gis Ferri{\`e}re, and Sylvie M{\'e}l{\'e}ard.
\newblock From individual stochastic processes to macroscopic models in
  adaptive evolution.
\newblock {\em Stochastic Models}, 24(sup1):2--44, 2008.

\bibitem{champagnat_invasion_2007}
Nicolas Champagnat and Sylvie M{\'e}l{\'e}ard.
\newblock Invasion and adaptive evolution for individual-based spatially
  structured populations.
\newblock {\em Journal of Mathematical Biology}, 55(2):147, 2007.

\bibitem{fournier_microscopic_2004}
Nicolas Fournier and Sylvie M{\'e}l{\'e}ard.
\newblock A microscopic probabilistic description of a locally regulated
  population and macroscopic approximations.
\newblock {\em The Annals of Applied Probability}, 14(4):1880--1919, 2004.

\bibitem{friedman2008partial}
Avner Friedman.
\newblock {\em Partial differential equations of parabolic type}.
\newblock Courier Dover Publications, 2008.

\bibitem{howe1989}
Henry~F Howe.
\newblock Scatter-and clump-dispersal and seedling demography: hypothesis and
  implications.
\newblock {\em Oecologia}, 79(3):417--426, 1989.

\bibitem{meleard_stochastic_2015}
Sylvie M{\'e}l{\'e}ard and Vincent Bansaye.
\newblock {\em Stochastic {{Models}} for {{Structured Populations}}: {{Scaling
  Limits}} and {{Long Time Behavior}}}, volume~1.
\newblock {Springer}, 2015.

\bibitem{pazy2012semigroups}
Amnon Pazy.
\newblock {\em Semigroups of linear operators and applications to partial
  differential equations}, volume~44.
\newblock Springer Science \& Business Media, 2012.

\bibitem{pizo2001}
Marco~A Pizo and Isaac Sim{\~a}o.
\newblock Seed deposition patterns and the survival of seeds and seedlings of
  the palm euterpe edulis.
\newblock {\em Acta Oecologica}, 22(4):229--233, 2001.

\bibitem{reichling2000}
Steven~B Reichling.
\newblock Group dispersal in juvenile brachypelma vagans (araneae,
  theraphosidae).
\newblock {\em The Journal of Arachnology}, 28(2):248--250, 2000.

\bibitem{roelly1986criterion}
Sylvie Roelly-Coppoletta.
\newblock A criterion of convergence of measure-valued processes: application
  to measure branching processes.
\newblock {\em Stochastics: An International Journal of Probability and
  Stochastic Processes}, 17(1-2):43--65, 1986.

\bibitem{soubeyrand2017group}
Samuel Soubeyrand and Anna-Liisa Laine.
\newblock When group dispersal and {A}llee effect shape metapopulation
  dynamics.
\newblock In {\em Annales Zoologici Fennici}, volume~54, pages 123--138.
  BioOne, 2017.

\bibitem{soubeyrand2014nonstationary}
Samuel Soubeyrand, Tom{\'a}{\v{s}} Mrkvicka, and Antti Penttinen.
\newblock A nonstationary cylinder--based model describing group dispersal in a
  fragmented habitat.
\newblock {\em Stochastic Models}, 30(1):48--67, 2014.

\bibitem{soubeyrand2011patchy}
Samuel Soubeyrand, Lionel Roques, J{\'e}r{\^o}me Coville, and Julien Fayard.
\newblock Patchy patterns due to group dispersal.
\newblock {\em Journal of Theoretical Biology}, 271(1):87--99, 2011.

\bibitem{soubeyrand2015evolution}
Samuel Soubeyrand, Ivan Sache, Fr{\'e}d{\'e}ric Hamelin, and Etienne~K Klein.
\newblock Evolution of dispersal in asexual populations: to be independent,
  clumped or grouped?
\newblock {\em Evolutionary Ecology}, 29:947--963, 2015.

\bibitem{takahashi2008}
Kazuaki Takahashi, Tadatoshi Shiota, Hiroo Tamatani, Masaru Koyama, and Izumi
  Washitani.
\newblock Seasonal variation in fleshy fruit use and seed dispersal by the
  japanese black bear (ursus thibetanus japonicus).
\newblock {\em Ecological Research}, 23(2):471--478, 2008.

\bibitem{wu2006elliptic}
Zhuoqun Wu, Jingxue Yin, and Chunpeng Wang.
\newblock {\em Elliptic and parabolic equations}.
\newblock World Scientific Publishing Company, 2006.

\end{thebibliography}
%\printbibliography
\doclicenseThis
\end{document}